\documentclass[12pt]{amsart}
\usepackage{amsmath,amsthm,amssymb}
\usepackage{dsfont}
\usepackage{mathrsfs} %use \mathscr
\usepackage[all,cmtip]{xy}
\usepackage{patchcmd}
\input{diagxy}
\usepackage[backref=page]{hyperref}
\usepackage{bm}
\usepackage{mathtools}
\usepackage{stmaryrd}
\usepackage{enumitem}
\usepackage{tikz}
\usepackage{tikzit}
\tikzstyle{black little dot}=[fill=black, draw=black, tikzit category=string diagram, shape=circle, radius=0.1cm]
\tikzstyle{medium box}=[fill=white, draw=black, shape=rectangle, tikzit category=string diagram, minimum width=1cm, minimum height=.75cm, tikzit draw=black, tikzit shape=rectangle, thick]
\tikzstyle{small box}=[fill=white, draw=black, shape=rectangle, tikzit category=string diagram, tikzit draw=black, tikzit shape=rectangle, minimum width=.5cm, minimum height=.5cm, thick]
\tikzstyle{small circle}=[fill=white, draw=black, shape=circle, tikzit category=string diagram, tikzit draw=black, tikzit shape=circle, thick]
\tikzstyle{circle}=[fill=white, draw=black, shape=circle, tikzit category=string diagram, tikzit draw=black, tikzit shape=circle, radius=.1cm, thick]
\tikzstyle{wide small box}=[fill=white, draw=black, shape=rectangle, tikzit category=string diagram, tikzit fill=white, tikzit draw=black, minimum width=1cm, minimum height=.5cm, thick]

\tikzstyle{wire}=[-, draw=black, tikzit draw=black, thick]
\tikzstyle{dashline}=[-, tikzit category=string diagram, thick, dashed]
\tikzstyle{dotsline}=[-, dotted]
\tikzstyle{cyanthinline}=[-, draw=cyan]

\usepackage[normalem]{ulem}
\usepackage[top=2.4cm, bottom=2.4cm, left=2cm, right=2cm]{geometry}
\usepackage{environ}
\allowdisplaybreaks
\makeatletter
\patchcommand\@starttoc{\begin{quote}}{\end{quote}}
\def\@tocline#1#2#3#4#5#6#7{\relax
  \ifnum #1>\c@tocdepth % then omit
  \else
    \par \addpenalty\@secpenalty\addvspace{#2}%
    \begingroup \hyphenpenalty\@M
    \@ifempty{#4}{%
      \@tempdima\csname r@tocindent\number#1\endcsname\relax
    }{%
      \@tempdima#4\relax
    }%
    \parindent\z@ \leftskip#3\relax \advance\leftskip\@tempdima\relax
    \rightskip\@pnumwidth plus4em \parfillskip-\@pnumwidth
    #5\leavevmode\hskip-\@tempdima
      \ifcase #1
       \or\or \hskip 1em \or \hskip 2em \else \hskip 3em \fi%
      #6\nobreak\relax
    \dotfill\hbox to\@pnumwidth{\@tocpagenum{#7}}\par
    \nobreak
    \endgroup
  \fi}
\makeatother

 \theoremstyle{plain}
 \newtheorem{thm}{Theorem}[section]
 \newtheorem{cor}[thm]{Corollary}
 \newtheorem{lem}[thm]{Lemma}
 
 \newtheorem{prop}[thm]{Proposition}

\theoremstyle{definition}
 \newtheorem{defn}[thm]{Definition}

\theoremstyle{remark}
 \newtheorem{rem}[thm]{Remark}

 \newtheorem{ter}[thm]{Terminology}

 \newtheorem{exam}[thm]{Example}

 \numberwithin{equation}{section}

\theoremstyle{plain}

 \DeclareMathOperator{\dom}{dom}
 \DeclareMathOperator{\cod}{cod}

\def\XXint#1#2#3{{\setbox0=\hbox{$#1{#2#3}{\int}$}
\vcenter{\hbox{$#2#3$}}\kern-.5\wd0}}

\newcommand{\R}{\mathds{R}}

\newcommand{\0}{\emptyset}

\DeclareMathAlphabet{\mathpzc}{OT1}{pzc}{m}{it}

\DeclarePairedDelimiterX\set[1]\{\}{%

#1
}
\newbox\gnBoxA
\newdimen\gnCornerHgt
\setbox\gnBoxA=\hbox{$\ulcorner$}
\global\gnCornerHgt=\ht\gnBoxA
\newdimen\gnArgHgt

\def\code #1{%
        \setbox\gnBoxA=\hbox{$#1$}%
        \gnArgHgt=\ht\gnBoxA%
        \ifnum \gnArgHgt<\gnCornerHgt
                \gnArgHgt=0pt%
        \else
                \advance \gnArgHgt by -\gnCornerHgt%
        \fi
        \raise\gnArgHgt\hbox{$\ulcorner$} \box\gnBoxA %
                \raise\gnArgHgt\hbox{$\urcorner$}}

\newcommand{\rest}{\upharpoonright}

\newcommand{\fun}{\longrightarrow}

\newcommand{\efun}{\longmapsto}
\newcommand{\sub}{\subseteq}

\newcommand{\mi}{\smallsetminus}

\DeclareMathOperator{\out}{out}

\newcommand{\dand}{\quad \text{and} \quad}

\newcommand{\cat}[1]{\textup{\textsf{#1}}}
\DeclareMathOperator{\cau}{\cat F_{\text{$m$}}}

\DeclareMathOperator{\pa}{pa}

\newcommand{\memph}[1]{{\color{magenta}\emph{#1}}}
\DeclareMathOperator{\intv}{Ref}

\newcommand{\moi}{\mathbf{I}}

\DeclareMathOperator{\inn}{in}
\newcommand{\stri}{\Gamma}
\newcommand{\DD}{\textup{DD}}
\DeclareMathOperator{\cut}{cut}
\DeclareMathOperator{\str}{st}
\DeclareMathOperator{\eval}{ev}
\DeclareMathOperator{\alp}{al}
\DeclareMathOperator{\botimes}{\bar \otimes}
\begin{document}

\title{A graphical construction of free Markov categories}

\author[Y. Yin]{Yimu Yin}
\address[Y. Yin]{Pasadena, California}
\email{yimu.yin@hotmail.com}

%\author[J. Zhang]{Jiji Zhang}
%\address[J. Zhang]{}
%\email{}

\thanks{The author would like to thank Tobias Fritz for very helpful discussions.}

\keywords{Markov category, surgery on string diagram, free category}
\subjclass[2020]{18M30, 18M35, 62A09, 62H22}

\begin{abstract}
We describe how to perform surgeries on Joyal-Street style diagrams and thereby construct free Markov categories. We also show that the construction is functorial over ordered directed acyclic graphs.
\end{abstract}

\maketitle

\tableofcontents

\section{Introduction}

We shall use the graphical language for symmetric monoidal categories, see \cite{selinger2010survey} for a quick overview. Let $\cat C$ be a symmetric monoidal category; the monoidal unit is denoted by $\moi$, the class of objects by $\cat C_0$, and the class of morphisms by $\cat C_1$.  The defining laws for a \memph{cocommutative monoidal comonoid} $A$ in $\cat C$ may be depicted as follows:
\begin{equation}\label{mon:law}
\begin{tikzpicture}[xscale = .65, yscale = .75, baseline=(current  bounding  box.center)]
\begin{pgfonlayer}{nodelayer}
		\node [style=none] (139) at (-1.625, 5.45) {};
		\node [style=none] (141) at (-1.625, 5.325) {};
		\node [style=none] (144) at (-2.1, 6.125) {};
		\node [style=none] (146) at (-2.1, 5.825) {};
		\node [style=none] (246) at (-0.925, 4.8) {};
		\node [style=none] (247) at (-0.925, 4.425) {};
		\node [style=none] (249) at (-0.2, 6.1) {};
		\node [style=none] (250) at (-0.2, 5.325) {};
		\node [style=none] (251) at (-1.15, 6.125) {};
		\node [style=none] (252) at (-1.15, 5.825) {};
		\node [style=none] (281) at (1.15, 5.175) {$=$};
		\node [style=none] (298) at (3.85, 5.45) {};
		\node [style=none] (299) at (3.85, 5.325) {};
		\node [style=none] (300) at (4.325, 6.125) {};
		\node [style=none] (301) at (4.325, 5.825) {};
		\node [style=none] (302) at (3.15, 4.8) {};
		\node [style=none] (303) at (3.15, 4.425) {};
		\node [style=none] (304) at (2.425, 6.1) {};
		\node [style=none] (305) at (2.425, 5.325) {};
		\node [style=none] (306) at (3.375, 6.125) {};
		\node [style=none] (307) at (3.375, 5.825) {};
		\node [style=none] (314) at (9.6, 5.175) {$=$};
		\node [style=none] (315) at (11.875, 5.175) {$=$};
		\node [style=none] (316) at (10.75, 6.1) {};
		\node [style=none] (317) at (10.75, 4.425) {};
		\node [style=none] (318) at (13.05, 6) {$\bullet$};
		\node [style=none] (319) at (13.05, 5.575) {};
		\node [style=none] (320) at (13.75, 5.05) {};
		\node [style=none] (321) at (13.75, 4.425) {};
		\node [style=none] (322) at (14.475, 6.1) {};
		\node [style=none] (323) at (14.475, 5.575) {};
		\node [style=none] (308) at (15.95, 6.25) {};
		\node [style=none] (309) at (17.375, 5.2) {};
		\node [style=none] (310) at (16.675, 4.65) {};
		\node [style=none] (311) at (16.675, 4.275) {};
		\node [style=none] (312) at (17.375, 6.25) {};
		\node [style=none] (313) at (15.95, 5.2) {};
		\node [style=none] (324) at (8.4, 6) {$\bullet$};
		\node [style=none] (325) at (8.4, 5.575) {};
		\node [style=none] (326) at (7.7, 5.05) {};
		\node [style=none] (327) at (7.7, 4.425) {};
		\node [style=none] (328) at (6.975, 6.1) {};
		\node [style=none] (329) at (6.975, 5.575) {};
		\node [style=none] (330) at (20.125, 6.1) {};
		\node [style=none] (331) at (20.125, 5.575) {};
		\node [style=none] (332) at (20.825, 5.05) {};
		\node [style=none] (333) at (20.825, 4.425) {};
		\node [style=none] (334) at (21.55, 6.1) {};
		\node [style=none] (335) at (21.55, 5.575) {};
		\node [style=none] (336) at (18.8, 5.175) {$=$};
		\node [style=none, label={[align=center]center: coassociativity\\$(\delta_A \otimes 1_A) \circ \delta_A =  (1_A \otimes \delta_A) \circ \delta_A$}] (337) at (1.25, 3.175) {};
		\node [style=none, label={[align=center]center:counitality\\$(1_A \otimes \epsilon_A) \circ \delta_A = 1_A = (\epsilon_A \otimes 1_A) \circ \delta_n$}] (338) at (10.95, 3.175) {};
		\node [style=none, label={[align=center]center:cocommutativity\\$\sigma_{AA} \circ \delta_A = \delta_A$}] (339) at (18.8, 3.175) {};
	\end{pgfonlayer}
	\begin{pgfonlayer}{edgelayer}
		\draw [style=wire] (141.center) to (139.center);
		\draw [style=wire] (146.center) to (144.center);
		\draw [style=wire] (247.center) to (246.center);
		\draw [style=wire] (250.center) to (249.center);
		\draw [style=wire] (252.center) to (251.center);
		\draw [style=wire, bend right=90, looseness=1.25] (141.center) to (250.center);
		\draw [style=wire, bend right=90, looseness=1.25] (146.center) to (252.center);
		\draw [style=wire] (299.center) to (298.center);
		\draw [style=wire] (301.center) to (300.center);
		\draw [style=wire] (303.center) to (302.center);
		\draw [style=wire] (305.center) to (304.center);
		\draw [style=wire] (307.center) to (306.center);
		\draw [style=wire, bend left=90, looseness=1.25] (299.center) to (305.center);
		\draw [style=wire, bend left=90, looseness=1.25] (301.center) to (307.center);
		\draw [style=wire] (311.center) to (310.center);
		\draw [style=wire] (317.center) to (316.center);
		\draw [style=wire] (319.center) to (318.center);
		\draw [style=wire] (321.center) to (320.center);
		\draw [style=wire] (323.center) to (322.center);
		\draw [style=wire, bend right=90, looseness=1.25] (319.center) to (323.center);
		\draw [style=wire, in=-90, out=90] (309.center) to (308.center);
		\draw [style=wire] (311.center) to (310.center);
		\draw [style=wire, in=-90, out=90, looseness=0.75] (313.center) to (312.center);
		\draw [style=wire, bend left=90, looseness=1.25] (309.center) to (313.center);
		\draw [style=wire] (325.center) to (324.center);
		\draw [style=wire] (327.center) to (326.center);
		\draw [style=wire] (329.center) to (328.center);
		\draw [style=wire, bend left=90, looseness=1.25] (325.center) to (329.center);
		\draw [style=wire] (331.center) to (330.center);
		\draw [style=wire] (333.center) to (332.center);
		\draw [style=wire] (335.center) to (334.center);
		\draw [style=wire, bend right=90, looseness=1.25] (331.center) to (335.center);
	\end{pgfonlayer}
\end{tikzpicture}
\end{equation}
Our convention is to draw a string diagram in the lower-left to upper-right direction. Therefore, above, the \memph{comultiplication} $\delta_A : A \fun A \otimes A$ is depicted as an upward fork
$\begin{tikzpicture}[xscale = .3, yscale = .3, baseline={([yshift=5pt]current bounding box.south)}]
   \begin{pgfonlayer}{nodelayer}
		\node [style=none] (324) at (8.4, 6.1) {};
		\node [style=none] (325) at (8.4, 5.575) {};
		\node [style=none] (326) at (7.7, 5.05) {};
		\node [style=none] (327) at (7.7, 4.425) {};
		\node [style=none] (328) at (6.975, 6.1) {};
		\node [style=none] (329) at (6.975, 5.575) {};
	\end{pgfonlayer}
	\begin{pgfonlayer}{edgelayer}
		\draw [style=wire] (325.center) to (324.center);
		\draw [style=wire] (327.center) to (326.center);
		\draw [style=wire] (329.center) to (328.center);
		\draw [style=wire, bend left=90, looseness=1.25] (325.center) to (329.center);
	\end{pgfonlayer}
\end{tikzpicture}$
and the \memph{counit} $\epsilon_A : A \fun \moi$ an upward dead-end
$\begin{tikzpicture}[xscale = .3, yscale = .3, baseline={([yshift=5pt]current bounding box.south)}]
   \begin{pgfonlayer}{nodelayer}
		\node [style=none] (326) at (7.7, 5.8) {$\bullet$};
		\node [style=none] (327) at (7.7, 4.425) {};
	\end{pgfonlayer}
	\begin{pgfonlayer}{edgelayer}
		\draw [style=wire] (327.center) to (326.center);
	\end{pgfonlayer}
\end{tikzpicture}$.

Suppose that for each  object $A \in \cat C$ there are distinguished arrows $\delta_A : A \fun A \otimes A$, called the \memph{duplicate} on $A$, and $\epsilon_A : A \fun \moi$, called the \memph{discard} on $A$, that form a cocommutative monoidal comonoid as depicted in (\ref{mon:law}) (the qualifier ``distinguished'' is needed because there may be other such entities). Moreover, they respect the monoidal product:
\begin{equation}\label{fins:frob}
\begin{tikzpicture}[xscale = .55, yscale = .55, baseline=(current  bounding  box.center)]
\begin{pgfonlayer}{nodelayer}
		\node [style=none] (42) at (7.725, -0.75) {$=$};
		\node [style=none] (108) at (4.275, 0.4) {};
		\node [style=none] (110) at (6.275, 0.4) {};
		\node [style=none] (111) at (5.275, -1.825) {};
		\node [style=none] (112) at (4.275, -0.35) {};
		\node [style=none] (113) at (6.275, -0.325) {};
		\node [style=none] (114) at (5.275, -1.075) {};
		\node [style=none] (117) at (9.225, 0.65) {};
		\node [style=none] (118) at (12.125, 0.65) {};
		\node [style=none] (119) at (10.05, -1.825) {};
		\node [style=none] (120) at (9.225, -0.625) {};
		\node [style=none] (121) at (10.85, -0.6) {};
		\node [style=none] (122) at (10.05, -1.225) {};
		\node [style=none] (124) at (10.825, 0.65) {};
		\node [style=none] (125) at (13.825, 0.65) {};
		\node [style=none] (126) at (12.95, -1.825) {};
		\node [style=none] (127) at (12.125, -0.625) {};
		\node [style=none] (128) at (13.825, -0.6) {};
		\node [style=none] (129) at (12.95, -1.25) {};
		\node [style=none] (131) at (19.65, -0.75) {$=$};
		\node [style=none] (134) at (18.2, -1.525) {};
		\node [style=none] (137) at (18.2, -0.05) {$\bullet$};
		\node [style=none] (138) at (20.25, -3) {$\epsilon_{A \otimes B} = \epsilon_{A} \otimes \epsilon_{B}$};
		\node [style=none] (145) at (21.025, -1.525) {};
		\node [style=none] (146) at (21.025, -0.05) {$\bullet$};
		\node [style=none] (148) at (22.275, -1.525) {};
		\node [style=none] (149) at (22.275, -0.05) {$\bullet$};
		\node [style=none] (151) at (9.25, -3) {$\delta_{A \otimes B} = (1_A \otimes \sigma_{BA} \otimes 1_B) \circ (\delta_{A} \otimes \delta_B)$};
	\end{pgfonlayer}
	\begin{pgfonlayer}{edgelayer}
		\draw [style=wire, bend right=90, looseness=1.25] (112.center) to (113.center);
		\draw [style=wire] (108.center) to (112.center);
		\draw [style=wire] (110.center) to (113.center);
		\draw [style=wire] (111.center) to (114.center);
		\draw [style=wire, bend right=90, looseness=1.25] (120.center) to (121.center);
		\draw [style=wire] (117.center) to (120.center);
		\draw [style=wire, in=90, out=-90, looseness=1.25] (118.center) to (121.center);
		\draw [style=wire] (119.center) to (122.center);
		\draw [style=wire, bend right=90, looseness=1.25] (127.center) to (128.center);
		\draw [style=wire, in=90, out=-90, looseness=1.25] (124.center) to (127.center);
		\draw [style=wire] (125.center) to (128.center);
		\draw [style=wire] (126.center) to (129.center);
		\draw [style=wire] (134.center) to (137.center);
		\draw [style=wire] (145.center) to (146.center);
		\draw [style=wire] (148.center) to (149.center);
	\end{pgfonlayer}
\end{tikzpicture}
\end{equation}
Following \cite{fritz2020synthetic}, we call $\cat C$ a \memph{Markov category} if the  discards are natural, that is, they are the component morphisms of the natural transformation between the identity functor $1_{\cat C}$ and the endofunctor that sends everything to $\moi$:
\begin{equation}\label{discar:nat}
\begin{tikzpicture}[xscale = .55, yscale = .55, baseline=(current  bounding  box.center)]
	\begin{pgfonlayer}{nodelayer}
		\node [style=none] (200) at (19.4, 1.75) {$=$};
		\node [style=none] (202) at (17.45, 1.525) {};
		\node [style=none] (205) at (17.45, 0.6) {};
		\node [style=small box] (207) at (17.45, 1.875) {$f$};
		\node [style=none] (208) at (17.45, 3.225) {$\bullet$};
		\node [style=none] (209) at (17.45, 2.3) {};
		\node [style=none] (210) at (21.15, 2.725) {$\bullet$};
		\node [style=none] (211) at (21.15, 1.1) {};
		\node [style=none] (212) at (19.4, -0.75) {$\epsilon_B \circ f = \epsilon_A$};
	\end{pgfonlayer}
	\begin{pgfonlayer}{edgelayer}
		\draw [style=wire] (202.center) to (205.center);
		\draw [style=wire] (208.center) to (209.center);
		\draw [style=wire] (210.center) to (211.center);
	\end{pgfonlayer}
\end{tikzpicture}
\end{equation}
However, it is unreasonable to also require duplicates to be natural, see the opening discussion in \cite[\S~10]{fritz2020synthetic}.

We may forego this discussion on the naturality of discards if the monoidal unit $\moi$ is actually terminal, since in that case the desired property would hold automatically; conversely, the naturality of discards, together with the other conditions stipulated above, implies that  $\moi$ is terminal, and hence $\epsilon_\moi  = 1_\moi$. Also note that $\delta_\moi = 1_\moi$  in any strict Markov category. See \cite[Remark~2.3]{fritz2020synthetic}.

Accordingly, a \memph{Markov functor} is a strong symmetric monoidal functor  $F : \cat C \fun \cat D$ between two Markov categories that preserves the distinguished  comonoids by way of matching duplicates and discards, that is, for all objects $A \in \cat C$, the diagrams
\begin{equation*}
  \bfig
      \Atriangle(0,0)/->`->`->/<400,400>[FA`FA \otimes FA`F(A \otimes A); \delta_{FA}`F\delta_A`\cong]
      \Atriangle(1500,0)/->`->`->/<250,400>[FA`\moi`F \moi; \epsilon_{FA}`F\epsilon_A`\cong]
      %\place(1000,-200)[\text{matching}]
  \efig
\end{equation*}
commute, where  the horizontal arrows are the structure isomorphisms in question; see \cite[Definition~10.14]{fritz2020synthetic}). Of course the discards are matched automatically, since $\moi$ is terminal. Denote by \cat{MarCat} the category of Markov categories and Markov functors. In this paper, for simplicity,  \cat{MarCat} is only implicitly treated as a $2$-category when we occasionally speak of \memph{comonoid equivalence}, that is, equivalence in \cat{MarCat}.

To formulate the notion of a free Markov category, as in \cite{selinger2010survey}, we first need to fix a suitable class of signatures (called ``tensor schemes'' in \cite{joyal1991geometry}), as follows.

A \memph{graph monoid} is a directed graph $G = (V(G), A(G), s, t)$ whose set of vertices come with the additional structure of a monoid. A \memph{graph monoid homomorphism} is a graph homomorphism that is also a homomorphism between the monoids in question.

For any set $S$, let $W(S)$ denote the free monoid of words over (the alphabet) $S$. It is a graph monoid, with the trivial directed graph structure (no arrows). More generally, a \memph{free graph monoid} is a triple $(G, S, g)$, where $G$ is a graph monoid and $g : S \fun V(G)$ is an injection, such that the monoid $V(G)$ is isomorphic to the free monoid $W(S)$ over $g$. We usually just say $G$ is free  graph monoid and leave $S$, $g$ implicit in notation when they are not needed in the discussion. Also it is harmless to treat $S$ as if it is a subset of $V(G)$.

A \memph{free $\DD$-graph monoid} is a free graph monoid $(G, S, g)$ such that, for each $a \in S$, there are distinguished arrows $a \fun aa$, called the \memph{duplicate} on $a$, and $a \fun \0$, called the \memph{discard} on $a$.

A Markov category $\cat C$ is \memph{free} over a free $\DD$-graph monoid $G$ if there is a graph monoid homomorphism $i : G \fun \cat C$ matching duplicates and discards such that, for any Markov category $\cat D$, any graph monoid homomorphism $f : G \fun  \cat D$ matching duplicates and discards factors through $i$, that is, there is a Markov functor $F : \cat C \fun \cat D$ such that $f = F \circ i$, and $F$ is unique up to a unique  monoidal natural isomorphism. The objects and morphisms in $i(G)$ are referred to as \memph{generators}.

Since Markov categories are symmetric monoidal categories, the same graphical language can be used, but  with the additional syntax depicted in (\ref{mon:law}), (\ref{fins:frob}), and (\ref{discar:nat}). This does not give the free  Markov category over a free graph monoid, though, at least not for string diagrams up to isomorphisms. We need a coarser equivalence relation to accommodate the equations in  (\ref{mon:law}), (\ref{fins:frob}), and (\ref{discar:nat}).

\begin{rem}
The surgery method described in this paper works  rather well in some applications where free Markov categories play an essential role, see \cite{yinzhang2021}. For other monoidal categories with more complex conditions, it is not that the method does not work, it is just that the equivalence relation so constructed may be hard to parse visually, and consequently working with a graphical language does not seem to offer advantages over a symbolic one, which in a sense defeats the purpose of graphical languages.

For the construction of free Markov categories, there is a more sophisticated combinatorial approach, see \cite{FriLiang2022}.
\end{rem}

\section{Preliminaries on Joyal-Street style diagrams}

The discussion below assumes familiarity with the first two chapters of  \cite{joyal1991geometry}. We begin by recalling some definitions therein.

A \memph{generalized topological graph} $\Gamma = (\Gamma, \Gamma_0)$ consists of a Hausdorff space $\Gamma$ and a finite subset $\Gamma_0 \sub  \Gamma$ such that the complement $\Gamma \mi  \Gamma_0$ is a one-dimensional manifold without boundary and with finitely many connected components. It is \memph{acyclic} if it has no circuits, that is, there is no embedding from the unit circle in $\R^2$ into $\Gamma$, in particular, each of the connected components in $\Gamma \mi  \Gamma_0$ is homeomorphic to the open unit interval $(0, 1) \sub \R$. We shall only consider acyclic graphs and hence may omit the qualifier for brevity.

An element of $\Gamma_0$ is called a \memph{node}. A connected component of $\Gamma \mi \Gamma_0$ is called an \memph{edge}. The set of edges of $\Gamma$ is denoted by $\Gamma_1$, and the set  $\overline \Gamma \mi \Gamma$ by $\partial \Gamma$, where $\overline \Gamma$ is the topological closure (compactification) of $\Gamma$. The points in $\partial \Gamma$ are referred to as the \memph{outer nodes} of $\Gamma$ (nodes in $\Gamma_0$ are also called \memph{inner nodes} for clarity). Note that $\partial \Gamma$ and $\Gamma_0$ together form the boundary of $\Gamma$.

The  compactification $\hat e$ of an edge $e$ is homeomorphic to the closed unit interval $[0, 1] \sub \R$. An edge $e$ is called \memph{pinned} if the inclusion $e \sub \Gamma$ can be extended to an embedding $\hat e  \fun \Gamma$, and \memph{half-loose} if it can only be extended to $\hat e$ minus one endpoint,  and \memph{loose} if neither of the previous two cases holds.

An \memph{orientation} of an edge $e$ is just a total ordering on the pair of endpoints of $\hat e$, where the \memph{source} $e(0)$ is the image of the first element under the embedding $\hat e \fun  \overline \Gamma$ and the
\memph{target} $e(1)$ that of the last element. We say that $\Gamma$ is \memph{progressive} if it carries a choice of orientation for
each of its edges. In that case, the \memph{input} $\inn(x)$ of a node $x \in \Gamma_0$ is the set of the oriented edges with target $x$ and the \memph{output} $\out(x)$ the set of those with source $x$. An inner node $x \in \Gamma_0$ is \memph{initial} if $\inn(x) = \0$ and \memph{terminal} if $\out(x) = \0$. If, in addition, $\Gamma$ is equipped with a choice of total orderings on $\inn(x)$ and $\out(x)$ for each $x \in \Gamma_0$ then it is  \memph{polarized}. We shall only consider polarized graphs and hence may omit the qualifier for brevity.

The \memph{domain} $\dom(\Gamma)$ of $\Gamma$ consists of the edges whose sources are outer nodes and the \memph{codomain} $\cod(\Gamma)$ those whose targets are outer nodes. The edges in $\dom(\Gamma)$, $\cod(\Gamma)$ are often identified with the corresponding outer nodes for convenience. We say that $\Gamma$ is \memph{anchored} if both $\dom(\Gamma)$ and $\cod (\Gamma)$ are equipped with a total ordering.

\begin{defn}
Let $\cat C$ be a symmetric strict monoidal category. A \memph{valuation} $v : \Gamma \fun \cat C$ is a pair of functions
\[
v_1 : \Gamma_0 \fun \cat C_1, \quad v_0 : \Gamma_1 \fun \cat C_0
\]
such that, for every node $x \in  \Gamma_0$ with $\inn(x) = (a_1, \ldots, a_n)$ and $\out(x) = (b_1, \ldots, b_m)$, $v_1(x)$ is of the form $\bigotimes_i v_0(a_i) \fun \bigotimes_i v_0(b_i)$; here if $n = 0$ or $m = 0$ then the monoidal product in question is just the monoidal unit $\moi$ of $\cat C$. The pair $(\Gamma, v)$ is called a  \memph{diagram}  in $\cat C$, and is denoted simply by $\Gamma$ when the context is clear. If $\Gamma$ is anchored then the  domain $\dom(\Gamma, v)$ of $(\Gamma, v)$ is the object $\bigotimes_{e_i \in \dom(\Gamma)} v_0(e_i)$, where the monoidal product is taken with respect to the ordering on $\dom(\Gamma)$, similarly for the codmain $\cod(\Gamma, v)$ of $(\Gamma, v)$.

An \memph{isomorphism of  diagrams} $\Gamma \fun \Omega$ is an isomorphism of the  graphs  that  is compatible with the valuations in the obvious way. If the graphs are anchored then the isomorphism is required to preserve the anchoring too.

A diagram in a free  graph monoid $(G, S, g)$ is formulated in the same way,  with the valuation $v : \Gamma \fun (G, S, g)$ mapping the edges and nodes into the alphabet $S$ and the set $A(G)$ of arrows, respectively.
\end{defn}

\begin{rem}\label{free:sys:mo:string}
Isomorphisms of graphs are homeomorphisms of one-dimensional topological spaces and, insofar as diagrams in  symmetric monoidal categories are concerned, they can be equivalently replaced by ambient isotopies of graphs embedded in $\R^4$, but not in lower dimensions. The reason is simply that, in the graphical language of symmetric monoidal categories,  we represent a symmetry as a braid, which is then forced by the axioms to be trivial, and indeed all braids in $\R^4$ are trivial. For instance, the crossing in the third diagram of (\ref{mon:law}) is  a projection in $\R^2$ of the symmetry braid in $\R^4$, and there is no need to depict, between the two strands, which one lies over which one, precisely because all the cases represent the same braid, that is, the trivial one. Thus, the intersection point therein is not a node of the graph, it is merely the (intentional) overlap of the shadows in $\R^2$ of the two edges, which do not intersect in $\R^4$ at all.
\end{rem}

\begin{thm}[{\cite[Theorem~2.3]{joyal1991geometry}}]\label{free:sym}
The isomorphism classes of anchored diagrams $(\Gamma, v)$ in the free  graph monoid $G$, denoted by $[\Gamma, v]$, form the free symmetric monoidal category $\cat F_s(G)$ over $G$.
\end{thm}

From here on we assume that $\Gamma$ is anchored.

Recall  that  $\Gamma_0$  may be ordered as follows: $x < y$ if and only if there is a directed path $p$ in $\Gamma$ with $p(0) = x$ to $p(1) = y$; the thus formed  partially ordered set is denoted by $\dot\Gamma$.

An initial segment of $\dot \Gamma$ is referred to as a \memph{level} of $\Gamma$, with $\0$ the smallest one and $\dot \Gamma$ the largest one.

An edge $e$ is said to be \memph{cut} by a level $L$ when $e(0) \in L \cup \dom(\Gamma)$ and $e(1) \in (\dot \Gamma \mi L) \cup \cod(\Gamma)$.

Let $\cut(L)$ denote the set of edges cut by $L$; we have $\cut(\0) = \dom(\Gamma)$ and $\cut(\dot \Gamma)  = \cod(\Gamma)$. For a pair of levels $L \sub M$, the \memph{layer} $\Gamma[L, M]$ in $\Gamma$ is the subgraph of $\Gamma$ such that
\begin{itemize}
  \item its inner nodes are those in $M \mi L$,
  \item its pinned edges are those $e \in \Gamma_1$ with $e(0), e(1) \in M \mi L$,
  \item its loose and half-loose edges are those in $\cut(L) \cup \cut(M)$.
\end{itemize}
So $\cut(L) = \dom(\Gamma[L, M])$ and $\cut(M) = \cod(\Gamma[L, M])$, and the loose edges are exactly those in $\cut(L) \cap \cut(M)$, in particular, $\dom(\Gamma) = \dom(\Gamma[\0, L])$ and $\cod(\Gamma) = \cod(\Gamma[L, \dot \Gamma])$.

Any subgraph $\Gamma' \sub \Gamma$ inherits the orientations of the edges and the total orderings of $\inn_{\Gamma'}(x) \sub \inn_{\Gamma}(x)$, $\out_{\Gamma'}(x) \sub \out_{\Gamma}(x)$ for each $x \in \Gamma'_0$, and hence is  polarized. On the other hand, there is not a natural way to order $\dom(\Gamma') \mi \dom(\Gamma)$ or $\cod(\Gamma') \mi \cod(\Gamma)$, so $\Gamma'$ is not naturally anchored unless these two sets are empty.

\begin{rem}\label{order:raley:com}
Since a layer $\Gamma[L, M]$ is not naturally anchored, it cannot be used as is to produce a morphism in  $\cat F_s(G)$. If we choose total orderings on the domain and codomain then of course $\Gamma[L, M]$ may be turned into an anchored graph, although different choices in general bring about nonisomorphic results. In some situations, what  orderings we choose is immaterial,  as long as they are compatible. For instance, if $\Gamma_1$, $\Gamma_2$, and $\Gamma_3$ are only equipped with partial orderings on the domains and codomains then by writing
\[
[\Gamma, v] = [\Gamma_3, v_3] \circ [\Gamma_2, v_2] \circ [\Gamma_1, v_1]
\]
we mean that the partial orderings may be extended to compatible total orderings so to make the equality hold. In particular, this shall be what is meant when we write
\[
[\Gamma, v] = [\Gamma[M, \dot \Gamma], v_3] \circ [\Gamma[L, M], v_2] \circ [\Gamma[\0, L], v_1],
\]
where $v_1$, $v_2$, and $v_3$ are the induced valuations.
\end{rem}

\section{Surgeries on diagrams}

We continue to work with an anchored graph $\Gamma$.

\begin{defn}\label{def:full}
We say that a subgraph $\Gamma' \sub \Gamma$ is \memph{normal} if
\begin{itemize}
  \item no directed path in $\Gamma$ contains two distinct edges in $\Gamma'$ if one of them is loose in $\Gamma'$,
  \item the degree of every $x \in \Gamma'_0$ in $\Gamma'$ is the same as in $\Gamma$, in other words, all the edges at $x$ in $\Gamma$ are included in $\Gamma'$,
  \item every directed path  $p$ in $\Gamma$ with $p(0), p(1) \in \Gamma'_0$ belongs to  $\Gamma'$.
\end{itemize}
\end{defn}

For a subgraph $\Gamma' \sub \Gamma$, let $\Gamma'^{\flat}$ denote the subgraph obtained from $\Gamma'$ by deleting its loose edges.

Observe that if two subgraphs $\Gamma'$, $\Gamma''$ have the same set of nodes and both satisfy the second  condition above then $\Gamma'^{\flat}$, $\Gamma''^{\flat}$ must be equal. Also, in a normal subgraph $\Gamma'$ of $\Gamma$, if $x \in  \Gamma'_0$ is initial in $\Gamma'$ then it is initial in $\Gamma$, similarly if it is terminal.

\begin{lem}\label{full:lay}
A subgraph $\Gamma' \sub \Gamma$ is normal if and only if there is a layer $\Gamma[L, M]$ in $\Gamma$ such that $\Gamma' \sub \Gamma[L, M]$ and $\Gamma[L, M]^\flat = \Gamma'^{\flat}$.
\end{lem}
\begin{proof}
For the ``if'' direction, it is enough to show that $\Gamma[L, M]$ is normal. First, let $p$ be a directed path in $\Gamma$. Suppose for contradiction that $p$ contains two distinct edges $e$, $e'$ in $\Gamma[L, M]$, where $e$ is loose in $\Gamma[L, M]$. We have, in $\dot \Gamma$, either $e(1) \leq e'(0)$ or $e'(1) \leq e(0)$. For the first case, since $e$ is cut by $M$, we have $e(1) \in \dot \Gamma \mi M$ and hence $e'(0) \in \dot \Gamma \mi M$, which entails that  $e'$ cannot be in $\Gamma[L, M]$, contradiction. Similarly, for the second case, since $e$ is cut by $L$ as well, we have $e(0) \in L$ and hence $e'(1) \in  L$, so  $e'$ cannot be in $\Gamma[L, M]$, contradiction again. Next, if $p(0), p(1) \in \Gamma[L, M]_0 = M \mi L$ then clearly $p$ belongs to $\Gamma[L, M]$. Lastly, for any $x \in \Gamma[L, M]_0$, if $e \in \inn_\Gamma(x)$ is cut by $L$ then it belongs to $\Gamma[L, M]$ by definition, otherwise $e(0) \in M \mi L$ and hence $e$ still belongs to $\Gamma[L, M]$ by definition, similarly for $e \in \out_\Gamma(x)$. So $\Gamma[L, M]$ is normal.

For the ``only if'' direction, let $\Gamma' \sub \Gamma$ be normal. Let $L$ be the set of the nodes $x \in \dot \Gamma$ such that either $x \notin \Gamma'_0$ and $x < y$ for some $y \in \Gamma'_0$ or $x \leq  e(0) \in \dot \Gamma$ for some loose edge $e$ in $\Gamma'$. If the first case holds and $z < x$ then $z \notin \Gamma'_0$, for otherwise $z < x < y$ and hence $x \in \Gamma'_0$ too by the third condition in Definition~\ref{def:full}, contradiction. So $L$ is a level. Actually, if $x \leq  e(0) \in \dot \Gamma$ for some loose edge $e$ in $\Gamma'$  then $x \notin  \Gamma'_0$, for otherwise $x < e(0)$ and there would be a directed path in $\Gamma$ that contains $e$ and an edge $e'$ at $x$, and since $e' \in \Gamma'$ by the second condition in Definition~\ref{def:full}, this contradicts the first condition therein. So  $L \cap \Gamma'_0 = \0$. Now, let $M$ be the set of the nodes $x \in \dot \Gamma$ such that $x \leq y$ for some $y \in \Gamma'_0$ or $x \leq e(0) \in \dot \Gamma$ for some loose edge $e$ in $\Gamma'$.  Clearly $M$ is a level and contains $L$. We have $M \mi L = \Gamma'_0$.

Since $\Gamma[L, M]^{\flat}$ is normal too, we have $\Gamma[L, M]^{\flat} = \Gamma'^{\flat}$. For any loose edge $e$ in $\Gamma'$, if $e(1) \in M$ then there would be a directed path in $\Gamma$ that contains two distinct edges in $\Gamma'$, one of which is $e$,  contradicting the first condition in Definition~\ref{def:full}, so  $e$ is cut by both $L$ and $M$, in other words, $e$ is a loose edge in $\Gamma[L, M]$. So $\Gamma' \sub \Gamma[L, M]$.
\end{proof}

Denote by $\Gamma'^{\sharp}$ the layer $\Gamma[L, M]$ constructed from $\Gamma'$ in the latter half of the proof above.

So normal subgraphs are almost the same as layers. The point, of course, is that they are more intuitive to work with.

Suppose that $\Gamma' \sub \Gamma$ is a normal subgraph. Let $\phi : \Omega \fun \Gamma'$ be an isomorphism of  graphs. Let $\Lambda$ be a  graph and
\[
\alpha : \dom(\Lambda) \fun \dom(\Omega), \quad \beta: \cod(\Lambda) \fun \cod(\Omega)
\]
bijections. We can graft $\Lambda$ onto $\Gamma$ by identifying the edges in $\dom(\Lambda)$, $\cod(\Lambda)$ with the corresponding edges in $\Gamma'$ --- for this to work, the loose edges in $\Gamma'$ need to be duplicated, with one copy each in $\dom(\Gamma')$ and $\cod(\Gamma')$ ---  and then cut out the rest of $\Gamma'$ from $\Gamma$. We call this operation the \memph{$(\Omega / \Lambda, \phi, \alpha, \beta)$-surgery} on $\Gamma$ and write $\Gamma^\Omega_\Lambda(\phi, \alpha, \beta)$ for the resulting graph, where $\phi$, $\alpha$, and $\beta$ shall be dropped from the notation when the context is clear; the triple $(\Omega / \Lambda, \alpha, \beta)$ is referred to as the \memph{template} of the surgery.

Clearly there are induced bijections
\[
\alpha^\sharp : \dom(\Gamma^\Omega_\Lambda) \fun \dom(\Gamma), \quad \beta^\sharp:  \cod(\Gamma^\Omega_\Lambda) \fun \cod(\Gamma)
\]
and hence $\Gamma^\Omega_\Lambda$ is indeed naturally anchored; in particular, if $\Lambda$ is isomorphic to $\Omega$ (as polarized graphs) then $\Gamma^\Omega_\Lambda$ is isomorphic to $\Gamma$ (as anchored polarized graphs). If the template is empty, that is, if $\Lambda = \Omega = \0$, then $\Gamma^\Omega_\Lambda = \Gamma$.

Let $G$ be a free graph monoid. Let $v : \Gamma \fun G$, $w : \Lambda \fun G$,  and $u : \Omega \fun G$ be valuations. Let $\phi : (\Omega, u) \fun (\Gamma', v')$ be an isomorphism of  diagrams, where  $v'$ is the restriction of $v$ to $\Gamma'$. Suppose that $u$, $w$ are compatible with respect to $\alpha$ and $\beta$, that is, $w(e) = (u \circ \alpha)(e)$ for $e \in \dom(\Lambda)$,  and similarly for $e \in \cod(\Lambda)$. Since $\Gamma' \sub \Gamma$ is normal, there is an induced valuation $v_w : \Gamma^\Omega_\Lambda \fun G$ fusing $v$, $w$ together such that $v_w(e) = (v \circ \alpha^\sharp)( e)$ for $e \in \dom(\Gamma^\Omega_\Lambda)$, and similarly for  $e \in \cod(\Gamma^\Omega_\Lambda)$. So \[
\dom(\Gamma^\Omega_\Lambda, v_w) = \dom(\Gamma, v) \dand \cod(\Gamma^\Omega_\Lambda, v_w) = \cod(\Gamma, v).
\]
The transition from the diagram $(\Gamma, v)$ to the diagram $(\Gamma^\Omega_\Lambda, v_w)$ is called the \memph{$((\Omega, u) / (\Lambda, w), \phi, \alpha, \beta)$-surgery} on $(\Gamma, v)$ with the \memph{template} $((\Omega, u) / (\Lambda, w), \alpha, \beta)$; again,  we usually only show $\Omega / \Lambda$ in notation.

Observe that a surgery can be reversed (by another surgery) in the sense that the resulting diagram is isomorphic to the one we started with.

\begin{lem}\label{sur:decom}
Let $\Omega$, $\Lambda$ be as above.  An anchored diagram $(\Xi, t)$ is the result of a surgery on  $(\Gamma, v)$ with a template  $\Omega / \Lambda$ if and only if there are
\begin{itemize}
  \item a normal subgraph $\Gamma' \sub \Gamma$ isomorphic to $\Omega$,
  \item total orderings that anchor $\Lambda$ and $\Gamma'$,
  \item a valuation $w: \Lambda \fun G$ with $\dom(\Lambda, w) = \dom(\Gamma', v')$ and $\cod(\Lambda, w) = \cod(\Gamma', v')$,
  \item morphisms $g$, $h$ in $\cat F_s(G)$
\end{itemize}
such that
\[
h \circ [\Gamma'^{\sharp}, v'^{\sharp}] \circ g = [\Gamma, v] \dand h \circ [\Lambda^{\sharp}, w^{\sharp}] \circ g = [\Xi, t],
\]
where $v'^{\sharp}$ is the restriction of $v$ to $\Gamma'^{\sharp}$, $\Lambda^{\sharp} = \Lambda \uplus (\Gamma'^{\sharp} \mi \Gamma')$,  and  $w^{\sharp} = w \uplus (v \rest (\Gamma'^{\sharp} \mi \Gamma'))$.
\end{lem}
\begin{proof}
The ``if'' direction is clear. For the ``only if'' direction, let $\Gamma' \sub \Gamma$ be the normal subgraph  in question. By Lemma~\ref{full:lay}, $\Gamma' \sub \Gamma'^{\sharp} = \Gamma[L, M]$ for some levels $L$, $M$. Then
\begin{gather*}
[\Gamma, v] = [\Gamma[M, \dot \Gamma], v_M] \circ [\Gamma'^{\sharp}, v'^{\sharp}] \circ [\Gamma[\0, L], v_L],\\
[\Xi, t] = [\Gamma[M, \dot \Gamma], v_M] \circ [\Lambda^{\sharp}, w^{\sharp}] \circ [\Gamma[\0, L], v_L].
\end{gather*}
where $v_M$, $v_L$ are the induced valuations; here we have used the convention outlined in Remark~\ref{order:raley:com}. The lemma follows.
\end{proof}

We shall often denote a diagram simply by $\Gamma$ when the valuation may be left implicit. By the same token, the same symbol may denote both the diagram and the underlying graph.

\begin{defn}
Let $T$ be a \memph{symmetric} set of  templates that contains the empty one; this just means that if $((\Omega, u) / (\Lambda, w), \alpha, \beta)$ is  in $T$ then $( (\Lambda, w) / (\Omega, u), \alpha^{-1}, \beta^{-1})$ is also   in $T$. A surgery  with a template in $T$ is referred to as a \memph{$T$-surgery}. Two anchored  diagrams $\Gamma$, $\Upsilon$ in $G$ are \memph{$T$-equivalent}, denoted by $\Gamma \leftrightsquigarrow_T \Upsilon$, if there is a sequence of anchored  diagrams
\[
(\Gamma = \Gamma_1, \ldots, \Gamma_n = \Upsilon)
\]
such that, for each $i$, there is a $T$-surgery on $\Gamma_i$ whose result is isomorphic to $\Gamma_{i+1}$.
\end{defn}

If $[\Gamma] = [\Upsilon]$, $[\Gamma'] = [\Upsilon']$, and $\Gamma \leftrightsquigarrow_T \Gamma'$ then clearly $\Upsilon \leftrightsquigarrow_T \Upsilon'$. So we may treat $T$-equivalence as a relation on the set of objects of $\cat F_s(G)$. It is also easy to see that if $[\Gamma] \leftrightsquigarrow_T [\Gamma']$ and $[\Upsilon] \leftrightsquigarrow_T [\Upsilon']$ then
\begin{itemize}
  \item $[\Upsilon] \otimes [\Gamma] \leftrightsquigarrow_T [\Upsilon'] \otimes [\Gamma']$,
  \item $[\Upsilon] \circ [\Gamma] \leftrightsquigarrow_T [\Upsilon'] \circ [\Gamma']$ if the compositions are defined.
\end{itemize}
So $T$-equivalence is indeed a monoidal congruence relation on $\cat F_s(G)$; we denote it by $\bm S_T$. It follows that the quotient $\cat F_s(G) / \bm S_T$ is a symmetric strict monoidal category.

Let $\bm  E_{T}$ be the monoidal congruence relation on $\cat F_s(G)$ generated by the pairs $([\Omega, u], [\Lambda, w])$, where $((\Omega, u) / (\Lambda, w), \alpha, \beta)$ runs through the templates in $T$ with any chosen orderings on the domains and codomains of $\Omega$, $\Lambda$ that are compatible with $\alpha$, $\beta$.

\begin{cor}\label{graph:small}
$\bm S_T = \bm  E_{T}$.
\end{cor}
\begin{proof}
That $\bm S_T \sub \bm  E_{T}$ is an easy consequence of Lemma~\ref{sur:decom}. That $\bm S_T \supseteq \bm  E_{T}$ follows from a routine induction on how $\bm  E_{T}$ is generated.
\end{proof}

\section{Constructing free Markov categories}

Suppose that $(G, S, g)$ is a free $\DD$-graph monoid.

The five equalities in (\ref{mon:law}) and (\ref{discar:nat}) give rise to the \memph{set of Markov templates} and the corresponding \memph{Markov surgeries}. We describe in detail what these  templates are. For coassociativity,
\begin{equation}\label{temp:coass}
\begin{tikzpicture}[xscale = .65, yscale = .6, baseline=(current  bounding  box.center)]
\begin{pgfonlayer}{nodelayer}
		\node [style=none] (139) at (-1.625, 5.45) {};
		\node [style=none] (141) at (-1.625, 5.325) {};
		\node [style=none, label={above : $H$}] (144) at (-2.1, 6.125) {};
		\node [style=none] (146) at (-2.1, 5.825) {};
		\node [style=none] (246) at (-0.925, 4.8) {};
		\node [style=none, label={below : $T$}] (247) at (-0.925, 4.425) {};
		\node [style=none, label={above : $A$}] (249) at (-0.2, 6.1) {};
		\node [style=none] (250) at (-0.2, 5.325) {};
		\node [style=none, label={above : $C$}] (251) at (-1.15, 6.125) {};
		\node [style=none] (252) at (-1.15, 5.825) {};
		\node [style=none] (281) at (3.4, 5.175) {$=$};
		\node [style=none] (298) at (6.1, 5.45) {};
		\node [style=none] (299) at (6.1, 5.325) {};
		\node [style=none, label={above : $A$}] (300) at (6.575, 6.125) {};
		\node [style=none] (301) at (6.575, 5.825) {};
		\node [style=none] (302) at (5.4, 4.8) {};
		\node [style=none, label={below : $T$}] (303) at (5.4, 4.425) {};
		\node [style=none, label={above : $H$}] (304) at (4.675, 6.1) {};
		\node [style=none] (305) at (4.675, 5.325) {};
		\node [style=none, label={above : $C$}] (306) at (5.625, 6.125) {};
		\node [style=none] (307) at (5.625, 5.825) {};
		\node [style=none] (338) at (-3.425, 5.175) {$=$};
		\node [style=none] (339) at (-4.675, 5.175) {$\Lambda_a$};
		\node [style=none] (340) at (2.075, 5.175) {$\Omega_a$};
	\end{pgfonlayer}
	\begin{pgfonlayer}{edgelayer}
		\draw [style=wire] (141.center) to (139.center);
		\draw [style=wire] (146.center) to (144.center);
		\draw [style=wire] (247.center) to (246.center);
		\draw [style=wire] (250.center) to (249.center);
		\draw [style=wire] (252.center) to (251.center);
		\draw [style=wire, bend right=90, looseness=1.25] (141.center) to (250.center);
		\draw [style=wire, bend right=90, looseness=1.25] (146.center) to (252.center);
		\draw [style=wire] (299.center) to (298.center);
		\draw [style=wire] (301.center) to (300.center);
		\draw [style=wire] (303.center) to (302.center);
		\draw [style=wire] (305.center) to (304.center);
		\draw [style=wire] (307.center) to (306.center);
		\draw [style=wire, bend left=90, looseness=1.25] (299.center) to (305.center);
		\draw [style=wire, bend left=90, looseness=1.25] (301.center) to (307.center);
	\end{pgfonlayer}
\end{tikzpicture}
\end{equation}
where the subscript ``$a$'' indicates the element in the alphabet $S$ as well as the corresponding duplicate in question, the ordering on the output of each node is given from left to right,  and the uppercase letters indicate what the bijections $\alpha$, $\beta$ are (we use letters instead of numerals to avoid unintended orderings). For left counitality,
\begin{equation}\label{temp:counit}
\begin{tikzpicture}[xscale = .65, yscale = .4, baseline=(current  bounding  box.center)]
\begin{pgfonlayer}{nodelayer}
		\node [style=none] (314) at (8.7, 5.675) {$=$};
		\node [style=none, label={above : $H$}] (316) at (4.5, 6.6) {};
		\node [style=none, label={below : $C$}] (317) at (4.5, 4.925) {};
		\node [style=none] (324) at (11.525, 6.5) {$\bullet$};
		\node [style=none] (325) at (11.525, 5.575) {};
		\node [style=none] (326) at (10.825, 5.05) {};
		\node [style=none, label={below : $C$}] (327) at (10.825, 4.425) {};
		\node [style=none, label={above : $H$}] (328) at (10.1, 6.6) {};
		\node [style=none] (329) at (10.1, 5.575) {};
		\node [style=none] (330) at (7.6, 5.675) {$\Omega_a$};
		\node [style=none] (331) at (2.15, 5.675) {$\Lambda_a$};
		\node [style=none] (332) at (3.15, 5.675) {$=$};
	\end{pgfonlayer}
	\begin{pgfonlayer}{edgelayer}
		\draw [style=wire] (317.center) to (316.center);
		\draw [style=wire] (325.center) to (324.center);
		\draw [style=wire] (327.center) to (326.center);
		\draw [style=wire] (329.center) to (328.center);
		\draw [style=wire, bend left=90, looseness=1.25] (325.center) to (329.center);
	\end{pgfonlayer}
\end{tikzpicture}
\end{equation}
where the dot indicates that the discard $\epsilon_a$ is assigned to the terminal node, similarly for right counitality. For cocommutativity,
\begin{equation}\label{temp:cocom}
\begin{tikzpicture}[xscale = .65, yscale = .4, baseline=(current  bounding  box.center)]
\begin{pgfonlayer}{nodelayer}
		\node [style=none] (314) at (9.2, 5.675) {$=$};
		\node [style=none, label={above : $H$}] (324) at (12.025, 6.625) {};
		\node [style=none] (325) at (12.025, 5.575) {};
		\node [style=none] (326) at (11.325, 5.05) {};
		\node [style=none, label={below : $A$}] (327) at (11.325, 4.425) {};
		\node [style=none, label={above : $C$}] (328) at (10.6, 6.6) {};
		\node [style=none] (329) at (10.6, 5.575) {};
		\node [style=none] (330) at (8.1, 5.675) {$\Omega_a$};
		\node [style=none] (331) at (2.15, 5.675) {$\Lambda_a$};
		\node [style=none] (332) at (3.15, 5.675) {$=$};
		\node [style=none, label={above : $C$}] (333) at (5.95, 6.625) {};
		\node [style=none] (334) at (5.95, 5.575) {};
		\node [style=none] (335) at (5.25, 5.05) {};
		\node [style=none, label={below : $A$}] (336) at (5.25, 4.425) {};
		\node [style=none, label={above : $H$}] (337) at (4.525, 6.6) {};
		\node [style=none] (338) at (4.525, 5.575) {};
	\end{pgfonlayer}
	\begin{pgfonlayer}{edgelayer}
		\draw [style=wire] (325.center) to (324.center);
		\draw [style=wire] (327.center) to (326.center);
		\draw [style=wire] (329.center) to (328.center);
		\draw [style=wire, bend left=90, looseness=1.25] (325.center) to (329.center);
		\draw [style=wire] (334.center) to (333.center);
		\draw [style=wire] (336.center) to (335.center);
		\draw [style=wire] (338.center) to (337.center);
		\draw [style=wire, bend left=90, looseness=1.25] (334.center) to (338.center);
	\end{pgfonlayer}
\end{tikzpicture}
\end{equation}
This may seem different from how cocommutativity is depicted in (\ref{mon:law}). As we have explained in Remark~\ref{free:sys:mo:string}, the crossing is an artifact of representing graphs on a plane and may be dispensed with if we spell out the orderings, which is all that really matters. Finally, for naturality of discard,
\begin{equation}\label{temp:discar}
\begin{tikzpicture}[xscale = .55, yscale = .5, baseline=(current  bounding  box.center)]
\begin{pgfonlayer}{nodelayer}
		\node [style=none] (314) at (7.525, 5.675) {$=$};
		\node [style=none, label={below : $1$}] (324) at (2.15, 5) {};
		\node [style=none] (325) at (2.15, 6.55) {$\bullet$};
		\node [style=none] (328) at (8.9, 6.075) {};
		\node [style=none] (329) at (8.9, 7.1) {$\bullet$};
		\node [style=none] (330) at (6.425, 5.675) {$\Omega_f$};
		\node [style=none] (331) at (-0.325, 5.675) {$\Lambda_f$};
		\node [style=none] (332) at (0.675, 5.675) {$=$};
		\node [style=none, label={below : $1$}] (337) at (8.9, 4.325) {};
		\node [style=none] (338) at (8.9, 5.35) {};
		\node [style=wide small box] (339) at (9.625, 5.75) {$f$};
		\node [style=none] (340) at (10.35, 6.075) {};
		\node [style=none] (341) at (10.35, 7.1) {$\bullet$};
		\node [style=none, label={below : $k$}] (342) at (10.35, 4.325) {};
		\node [style=none] (343) at (10.35, 5.35) {};
		\node [style=none] (344) at (9.625, 6.75) {$\cdots$};
		\node [style=none] (345) at (9.625, 4.7) {$\cdots$};
		\node [style=none, label={below : $k$}] (346) at (3.65, 5) {};
		\node [style=none] (347) at (3.65, 6.55) {$\bullet$};
		\node [style=none] (348) at (2.9, 5.75) {$\cdots$};
	\end{pgfonlayer}
	\begin{pgfonlayer}{edgelayer}
		\draw [style=wire] (325.center) to (324.center);
		\draw [style=wire] (329.center) to (328.center);
		\draw [style=wire] (338.center) to (337.center);
		\draw [style=wire] (341.center) to (340.center);
		\draw [style=wire] (343.center) to (342.center);
		\draw [style=wire] (347.center) to (346.center);
	\end{pgfonlayer}
\end{tikzpicture}
\end{equation}
where the subscript ``$f$'' indicates the arrow  in question and the numerals indicate the ordering on the input of $f$ as well as what the bijection $\alpha$ is (there is no need to depict $\beta$ as the graphs have the empty codomain). We also include the versions that switch the roles of $\Lambda$, $\Omega$ so that the set becomes symmetric. Denote the resulting monoidal congruence relation on $\cat F_s(G)$ by $\bm S_M$.

Let $a = a_1 \ldots a_n$ be an object in $\cat F_s(G)$, where $a_i \in S$. Let $\Gamma$ be the diagram that joins those of $\bigotimes_i \delta_{a_i}$, $(\bigotimes_i 1_{a_i}) \otimes (\bigotimes_i 1_{a_i})$ together by connecting the codomain of the former with the domain of the latter. Of course $\Gamma$ is not unique, as it depends on the order in which the strings are connected. On the other hand, in light of (\ref{temp:cocom}), all such $\Gamma$ represent the same  morphism in the quotient category $\cat F_s(G) / \bm S_M$, which is denoted by $\delta_a$. Also set $\epsilon_a = \bigotimes_i \epsilon_{a_i}$ in $\cat F_s(G) / \bm S_M$. Then the equalities (\ref{mon:law}), (\ref{fins:frob}), and (\ref{discar:nat}) hold in $\cat F_s(G) / \bm S_M$ by construction, in other words, $\cat F_s(G) / \bm S_M$ is a Markov category.

Let $i_s : G \fun \cat F_s(G)$ be the obvious graph monoid homomorphism (see the proofs of \cite[Theorems~1.2, 2.3]{joyal1991geometry} for details); note that either one of the diagrams in (\ref{temp:cocom}) may be designated as the duplicate on $a$ in $\cat F_s(G)$, but the choice is manifestly unnatural, so we do not treat $i_s$ as a graph monoid homomorphism matching duplicates and discards. On the other hand,
\[
i_m = (- / \bm S_M ) \circ i_s: G \fun \cat F_s(G) / \bm S_M
\]
is a graph monoid homomorphism matching duplicates and discards.

Now, let $\cat D$ be a Markov category and $f : G \fun  \cat D$ a graph monoid homomorphism matching duplicates and discards. By Theorem~\ref{free:sym}, there is a strong  symmetric monoidal functor
\[
F_s : \cat F_s(G) \fun \cat D \quad \text{with} \quad f = F_s \circ i_s.
\]
Since $\cat D$ is a  Markov category, by Corollary~\ref{graph:small}, we must have $F_s([\Upsilon]) = F_s([\Gamma])$ if $[\Upsilon] /  \bm S_M =  [\Gamma] /  \bm S_M$. This yields a Markov functor $F_m : \cat F_s(G) /  \bm S_M \fun \cat D$ with $f = F_m \circ i_m$. In summary, we have a commutative diagram
\begin{equation}\label{free:mark:dia}
  \bfig
      \Vtrianglepair(0,0)/<-`->`<-`<-`<-/<700,400>[\cat F_s(G) / \bm S_M`\cat F_s(G)`\cat D`G; - / \bm S_M`F_s`i_m`i_s`f]
      \morphism(0,400)|a|/{@{->}@/^3em/}/<1400,0>[\cat F_s(G) / \bm S_M`\cat D;F_m]
  \efig
\end{equation}
Since $F_s$ is unique up to a unique  monoidal natural isomorphism, so is $F_m$. In conclusion, we have shown:

\begin{thm}\label{free:markov:quot}
The quotient category $\cat F_s(G) / \bm S_M$ is the free Markov category over $G$.
\end{thm}

More informatively, this means that the graphical language of  symmetric monoidal categories over $G$, with the additional syntax depicted in (\ref{mon:law}), (\ref{fins:frob}), and (\ref{discar:nat}), up to Markov surgeries  (instead of isomorphisms as in free symmetric monoidal categories and many other cases discussed in \cite{selinger2010survey}) forms the free  Markov category over $G$, which is strict. Also, as pointed out after the proof of \cite[Theorem~1.2]{joyal1991geometry}, there are a unique symmetric strict  monoidal functor $F_s$ and hence such a unique $F_m$ that make (\ref{free:mark:dia}) commute.

\begin{rem}\label{graph:lan:mar}
In terms of coherence, freeness may be reformulated as follows. A well-formed
equation between morphisms in the symbolic language of symmetric monoidal categories
follows from the axioms of Markov categories if and only if it holds, up to Markov surgeries, in the graphical language of symmetric monoidal categories.
\end{rem}

A node $x$ of an anchored diagram $(\Gamma, v)$ in $G$ is \memph{quasi-terminal} if either it is terminal or every maximal directed path $p$ in $\Gamma$ with $p(0) = x$ ends at a terminal node or, in case that $v(x)$ is a duplicate, this is so for all such paths through one of the prongs.  Denote the set of quasi-terminal nodes of $\Gamma$ by $\Delta_\Gamma$ and its complement by  $\tilde \Delta_\Gamma = \Gamma_0 \mi \Delta_\Gamma$.

A \memph{splitter path} in $\Gamma$  is a  concatenation of two directed paths joined at the starting nodes. Denote by $P_\Gamma$ the set of directed paths that end in  $\cod(\Gamma)$ and by $S_\Gamma$ the set of splitter paths between edges in $\cod(\Gamma)$.

For the next two lemmas, suppose that  $(\Gamma, v)$, $(\Upsilon, w)$ are anchored diagrams belonging to the same $\bm S_M$-congruence class.

\begin{lem}\label{decor:mat}
There are
\begin{itemize}
 \item a bijection $\pi : \tilde \Delta_\Gamma \fun \tilde \Delta_{\Upsilon}$ compatible with the valuations, that is, $v(x) = w(\pi(x))$ for all $x \in \tilde \Delta_\Gamma$;
 \item a bijection $\dot \pi :  P_\Gamma \fun P_{\Upsilon}$  compatible with $\pi$, that is, $\pi$ restricts to a bijection between the nodes in $\tilde \Delta_\Gamma$ belonging to $p \in P_\Gamma$ and those in $\tilde \Delta_\Upsilon$ belonging to $\dot \pi(p)$,
 \item a bijection $\ddot \pi : S_\Gamma \fun S_{\Upsilon}$ compatible with $\pi$.
\end{itemize}

\end{lem}
\begin{proof}
All these are immediate by an induction on the least number of Markov surgeries required to get from $\Gamma$ to $\Upsilon$ and inspection of the Markov templates.
\end{proof}

Call the diagram $\Gamma$ \memph{Markov minimal} if every node in $\Delta_\Gamma$ is terminal; in that case, surgeries with the templates  (\ref{temp:counit}), (\ref{temp:discar}) can no longer be applied in the $\Omega$-to-$\Lambda$ direction. It follows that there is a Markov minimal anchored diagram in every $\bm S_M$-congruence class.

We say that $\Upsilon$ is \memph{Markov congruent} to $\Gamma$ if they become isomorphic upon surgeries with the templates (\ref{temp:coass}), (\ref{temp:cocom}).

\begin{lem}\label{mar:min:com}
If $\Gamma$, $\Upsilon$ are Markov minimal  then they are Markov congruent.
\end{lem}
\begin{proof}
We show this by an induction on the least number $n$ of Markov surgeries needed to beget $\Upsilon$ from $\Gamma$. The base case $n = 1$ is clear. For the inductive step, let $\sigma_1, \ldots, \sigma_n$ be a sequence of Markov surgeries that begets $\Upsilon$ from $\Gamma$ and $\Gamma_i$ the result of $\sigma_i$. Let $k_i$ be the size of $\Delta_{\Gamma_i}$. We may assume $k_i > 0$ for all $i < n$, for otherwise the claim would follow from the inductive hypothesis. Let $l < n-1$ be the least number such that $k_{l}, \ldots, k_n$ are strictly decreasing. So, for every $i > l$,  $\sigma_i$ is a surgery with the template  (\ref{temp:counit}) or (\ref{temp:discar}) applied in the $\Omega$-to-$\Lambda$ direction. Now, if $k_{l-1} = k_l$ then $\sigma_l$ is a surgery with the template (\ref{temp:coass}) or (\ref{temp:cocom}), in which case we can delete $\sigma_l$ and modify the surgeries $\sigma_i$, $i > l$, accordingly so to obtain a shorter sequence that begets $\Upsilon$ from $\Gamma$, contradicting the choice of $n$. Similarly, if $k_{l-1} < k_l$ then $\sigma_l$ is a surgery with the template (\ref{temp:counit}) or (\ref{temp:discar}) applied in the $\Lambda$-to-$\Omega$ direction, in which case we can delete $\sigma_l$ together with $\sigma_{l'}$ for some $l < l' \leq n$ and obtain a contradiction. So $k_i = 0$ for all $i$.  The lemma follows.
\end{proof}

\section{Effects in a Markov category}

Generalizing the construction of a causal conditional in \cite[\S~4]{Fong:thesis}, we shall define a class of morphisms in an arbitrary Markov category $\cat M$ that will be of central interest.

We first assume that $\cat M$ is \memph{straight}; this just means that $\cat M$ is strict and the monoid $(\cat M_0, \otimes, \moi)$ has no idempotents or elements of finite order  other than $\moi$; these two conditions are met in many natural situations where the results of this paper are intended for application.  Note that the second condition holds if and only if $w^n \neq w^m$ for all $w \neq \moi$ and all $n \neq m > 0$; here $w^n$ is a shorthand for  the  monoidal product of $n$ copies of $w$ itself; this includes the empty product $w^0 = \moi$.

\begin{defn}
Let $W = ( w_1, \ldots, w_n )$ be a sequence of objects  and $w = \bigotimes_i w_i$. A morphism $w \fun v$ in $\cat M$ is called a \memph{multiplier} on $(W, v)$  if it is generated from the duplicates, discards, symmetries, and identities on $w_i$, $1 \leq i \leq n$.
\end{defn}

\begin{rem}\label{mul:pow}
The straightness of $\cat M$ and (\ref{mon:law}) or, more intuitively, the coherence of  the graphical language for Markov categories that has been established above, guarantee that if $n = 1$ then there is a unique multiplier on $(w, w^m)$, which is denoted by $\iota_{w \rightarrow w^m}$. In that case, the multiplier may be depicted as a diagram that contains exactly one terminal node when $m = 0$, no node when $m = 1$, and $i$ nodes for $1 \leq i \leq m -1$ when $m > 1$, where $i$ is chosen to be conducive to  the situation at hand.
\end{rem}

\begin{rem}\label{uni:pro}
Suppose that $(\cat M_0, \otimes, \moi)$ is indeed a free monoid.

If every $w_i$ is in the alphabet and $w_i \neq w_j$ for all $i \neq j$ then there is again a unique multiplier $\iota_{\bar w \rightarrow v}$ on $(\bar w, v)$. For instance, if  $w = w_1 w_2 w_3$ and $v = w_1^2  w_2   w_1  w_2$ then $\iota_{\bar w \rightarrow v}$ may be depicted as the Markov minimal diagram
$\begin{tikzpicture}[xscale = .3, yscale = .3, baseline={([yshift=8pt]current bounding box.south)}]
\begin{pgfonlayer}{nodelayer}
		\node [style=none] (108) at (4.275, 0.15) {};
		\node [style=none] (110) at (7.325, 0.15) {};
		\node [style=none] (111) at (5.275, -1.575) {};
		\node [style=none] (112) at (4.275, -0.1) {};
		\node [style=none] (113) at (6.275, -0.65) {};
		\node [style=none] (114) at (5.275, -0.95) {};
		\node [style=none] (117) at (6.175, 0.15) {};
		\node [style=none] (118) at (8.8, 0.15) {};
		\node [style=none] (119) at (7.975, -1.575) {};
		\node [style=none] (120) at (7.125, -0.6) {};
		\node [style=none] (121) at (8.775, -0.35) {};
		\node [style=none] (122) at (7.975, -1.075) {};
		\node [style=none] (124) at (10.2, -1.55) {};
		\node [style=none] (150) at (10.2, 0.025) {$\bullet$};
		\node [style=none] (151) at (5.275, 0.15) {};
		\node [style=none] (152) at (4.4, -1.5) {$w_1$};
		\node [style=none] (153) at (7, -1.55) {$w_2$};
		\node [style=none] (154) at (9.2, -1.5) {$w_3$};
	\end{pgfonlayer}
	\begin{pgfonlayer}{edgelayer}
		\draw [style=wire, in=-135, out=-90] (112.center) to (113.center);
		\draw [style=wire] (108.center) to (112.center);
		\draw [style=wire, in=45, out=-90, looseness=1.25] (110.center) to (113.center);
		\draw [style=wire, in=270, out=90] (111.center) to (114.center);
		\draw [style=wire, in=-90, out=-60, looseness=1.25] (120.center) to (121.center);
		\draw [style=wire, in=120, out=-90] (117.center) to (120.center);
		\draw [style=wire, in=90, out=-90, looseness=1.25] (118.center) to (121.center);
		\draw [style=wire] (119.center) to (122.center);
		\draw [style=wire] (124.center) to (150.center);
		\draw [style=wire] (114.center) to (151.center);
	\end{pgfonlayer}
\end{tikzpicture}$, where how the  duplicates in the trident are arranged, how the edges at the nodes are ordered, how the copies of the same object in the codomain are ordered, and so on, can all be left unspecified without any ambiguity.

If $w_i = w_j$ for some $i \neq j$ then multipliers may not unique. In fact, this is so even if $w_i \neq w_j$ for all $i \neq j$, because there may be ``algebraic relations'' among $w_i$, $1 \leq i \leq n$, for example, $w_1w_2 = w_3$, unless every $w_i$ is in the alphabet, which then is the situation just discussed.
\end{rem}

A \memph{loop} in a directed graph  $G = (V(G), A(G), s, t)$ is an arrow in  $A(G)$ whose source and target coincide. Note that  $G$ could still be acyclic even if it has loops, because a cycle in $G$ does not contain loops by definition.  We shall work with directed graphs $G$ such that there is exactly one loop at each vertex in $G$, which is referred to as the \memph{identity loop} on $v$ and is denoted by $1_v$, for reasons that will become clear.

The category \cat{FinDAG} has finite directed acyclic graphs with identity loops as objects and graph homomorphisms between them as morphisms.

We have a more flexible notion of morphisms in \cat{FinDAG}, because  arrows can be collapsed due to the presence of identity loops. For instance, under the standard definition, there can be no morphisms from any directed graph with a nonempty set of arrows to the directed graph with exactly one vertex; on the other hand, the one-vertex graph is the terminal object in \cat{FinDAG} as defined above, that is, every object in \cat{FinDAG} comes with exactly one morphism to it.

\begin{defn}\label{gen:condi}
Let $W = (w_1, \ldots, w_n)$ be a sequence of objects with $w_i \neq \moi$ for all $i$. For each $S \sub \{1, \ldots, n\}$, let $w_S = {\bigotimes_{i \in S}}  w_i$, where $w_\0 = \moi$. For each $i$, let $\kappa_{i} : w_{S_i} \fun w_i$ be a morphism with $i \notin S_i$; let $K$ be the set of these morphisms $\kappa_{i}$. We associate a directed graph $G$ with the pair $(W, K)$: the set $V(G)$ of vertices is $\{1, \ldots, n\}$  and there is an arrow $j  \rightarrow i$ if and only if  $j \in S_i$. Of course $G$ may or may not be cyclic. After adding identity loops, we do assume $G \in \cat{FinDAG}$. Among other things, this implies $S_i = \0$ for some $i$; in that case, $\kappa_i$ shall be depicted as $\begin{tikzpicture}[xscale = .3, yscale = .3, baseline={([yshift=6pt]current bounding box.south)}]
\begin{pgfonlayer}{nodelayer}
		\node [style=none] (124) at (9.825, -1.55) {$\bullet$};
		\node [style=none] (150) at (9.825, 0.025) {};
	\end{pgfonlayer}
	\begin{pgfonlayer}{edgelayer}
		\draw [style=wire] (124.center) to (150.center);
	\end{pgfonlayer}
\end{tikzpicture}$.

Let $S, T \sub V(G)$. Let $G_{S \rightarrow T}$ be the subgraph of $G$ that consists of all the vertices in $S \cup T$ and all the directed paths that end  in $T$ but do not \memph{travel toward} $S$, that is, do not pass through or end in $S$ (starting in $S$ is allowed). Note that, for every  $i \in V(G_{S \rightarrow T})$, if $i \notin S$ then its parents in $G$ are all in $G_{S \rightarrow T}$ as well and if $i \in S$ then it has no parents in $G_{S \rightarrow T}$.

Construct a diagram in $\cat M$  as follows. For each $i \in V(G_{S \rightarrow T})$, let $\bar w_i$ be the monoidal product of as many copies of $w_i$ as the number of children of $i$ in $G_{S \rightarrow T}$. Let $\stri_i$ be a  diagram of
\[
\begin{dcases*}
   \iota_{w_i \rightarrow \bar w_i w_i}                & if $i \in S \cap T$,\\
   \iota_{w_i \rightarrow \bar w_i}                            & if $i \in S \mi T$,\\
  \iota_{w_i \rightarrow \bar w_i w_i} \circ \kappa_i  & if $i \in T \mi S$,\\
  \iota_{w_i \rightarrow \bar w_i} \circ \kappa_i              & if $i \notin S \cup T$,
\end{dcases*}
\]
where in the last cases the diagram for $\kappa_i$ is the obvious one; note the extra copy of $w_i$ in the codomain of $\iota_{w_i \rightarrow \bar w_i w_i}$. Recall from Remark~\ref{mul:pow} that the multipliers used here are unique and there is no need to specify orderings for their codomains in the diagrams.  For each $j \in V(G_{S \rightarrow T})$, let $o_j$ be the number of edges in the codomain of  $\stri_j$  and $p_j$ the number of all the edges with the value $w_j$ in the domains of all the other diagrams $\stri_i$. Observe that  $o_j = p_j + 1$ if $j \in T$ and  $o_j = p_j$ in all other cases. So we may identify the corresponding edges and fuse these \memph{components} $\stri_i$ together into a single  diagram, denoted by $\stri_{[w_T \| w_S]_{(W, K)}}$, with the domain $w_S$ and the codomain $w_T$.

The diagram thus obtained is Markov minimal. It may not be unique up to isomorphisms, but is unique up to Markov congruence. So it represents a unique morphism, which is referred to as the \memph{$(W, K)$-effect} of $w_S$ on $w_T$ and is denoted by $[w_T \| w_S]_{(W, K)} : w_S \fun w_T$, or simply $[w_T]_{(W, K)}$ when $S = \0$.
\end{defn}

We shall write $[w_T \| w_S]$ simply as $[T \| S]$ when the context is clear. The subscript $(W, K)$ may be dropped from the notation too.

\begin{lem}\label{graph:split}
For every $i \in V(G_{S \rightarrow T}) \mi S$ there are a level $L$  of $\stri_{[T \| S]}$ and a subset $T_i$ of $V(G_{S \rightarrow T})$ containing $i$ such that
\begin{itemize}
  \item $V(G_{S \rightarrow T_i})$ does not contain any child of $i$ in $G_{S \rightarrow T}$,
  \item $\stri_{[T \| S]}[\0, L] = \stri_{[{T_i} \| S]}$ and $\stri_{[T \| S]}[L, \dot \stri_{[T \| S]}] = \stri_{[{T} \| {T_i}]}$.
\end{itemize}

\end{lem}
\begin{proof}
To begin with, note that the component $\Gamma_i \sub \stri_{[T \| S]}$ is a normal subgraph. Let  $L$ be the level  of $\stri_{[T \| S]}$ as constructed in the latter half of the proof of Lemma~\ref{full:lay}; more precisely, $L$ consists of those nodes $x \in \dot \stri_{[T \| S]}$ such that $x \leq y_i \in (\stri_i)_0$, where $y_i$ is the least node in $\dot \Gamma_i$, which exists if $(\Gamma_i)_0 \neq \0$, in particular, if $i \notin S$. Let $T_i \sub V(G_{S \rightarrow T})$ such that  $j \in T_i$ if and only if there is an edge in $\cut(L)$ with the value $w_j$. So $i \in T_i$, and no child $k$ of $i$ in $G_{S \rightarrow T}$ can be an ancestor of any $j \in T_i$ in $G$, in other words, $k \notin  V(G_{S \rightarrow T_i})$.  Let $S_i \sub S$ such that  $j \in S_i$ if and only if there is an edge in $\cut(\0) \cap \cut(L)$, that is, a loose edge of $\stri_{[T \| S]}[\0, L]$, with the value $w_j$. So  every vertex in $(T_i \cup S) \mi S_i$ is an ancestor of $i$ in $G_{S \rightarrow T}$ (vertices are ancestors of themselves).

Consider any $j \in V(G_{S \rightarrow T_i})$. Assume $j \notin S_i$.  If $j \notin T_i \cup S$ then it must be an ancestor of some $k \in T_i \mi S_i$ in $G_{S \rightarrow T_i}$. So, at any rate, $j$ is an ancestor of $i$ in $G_{S \rightarrow T}$. It follows that if $j \notin T_i$ then the component $\Gamma_j$ of $\stri_{[T \| S]}$ is entirely contained in $\stri_{[T \| S]}[\0, L]$  and, more importantly, we have $e(1) \in L$ for every $e \in (\Gamma_j)_1$ and hence $\Gamma_j$ may be computed with respect to $G_{S \rightarrow T_i}$ too, that is, it may be regarded as a component of $\stri_{[{T_i} \| S]}$. If $j \in T_i$ then we may choose a suitable diagram of the multiplier in question (recall the last sentence of Remark~\ref{mul:pow}) and thereby assume that there is exactly one edge in $\cut(L)$ with the value $w_j$. In that situation, the portion of the component $\Gamma_j$ of $\stri_{[T \| S]}$ contained in $\stri_{[T \| S]}[\0, L]$ is again the component $\Gamma_j$ of $\stri_{[{T_i} \| S]}$.

Now assume $j \in S_i$. Then, by the discussion above, any child of $j$ in $G_{S \rightarrow T_i}$ would be an ancestor of $i$ in $G_{S \rightarrow T}$, which contradicts the definition of $S_i$. So  the component $\Gamma_j$ of $\stri_{[T_i \| S]}$ depicts the identity of $w_j$, which of course is the same as the loose edge in $\cut(\0) \cap \cut(L)$.

Finally, if $w_k$ is the value of an edge of   $\stri_{[T \| S]}[\0, L]$ then $k$ must belong to $V(G_{S \rightarrow T_i})$. So we have shown $\stri_{[T \| S]}[\0, L] = \stri_{[{T_i} \| S]}$.

By the choice of $\Gamma_j$ for each $j \in T_i \mi S_i$ made above, we may identify  $\cut(L)$  with $T_i$. Then a similar analysis shows that $\stri_{[T \| S]}[L, \dot \stri_{[T \| S]}] = \stri_{[{T} \| {T_i}]}$.
\end{proof}

\begin{prop}\label{mar:fun:eff}
Let $\cat N$ be another straight Markov category and $F : \cat M \fun \cat N$ a strict Markov functor. Then
\[
F([T \| S]_{(W, K)}) = [T \| S]_{(F(W), F(K))}.
\]
\end{prop}
\begin{proof}
The directed acyclic graph associated with $(F(W), F(K))$ is also $G$ and hence we may take the diagrams $\stri_{[T \| S]_{(W, K)}}$, $\stri_{[T \| S]_{(F(W), F(K))}}$ to have the same underlying graph with different valuations in different categories.

The \memph{height} of a directed acyclic graph is the maximum number of edges in a directed path.  We proceed by induction on the height $h$ of $G_{S \rightarrow T}$.

The case $h = 0$ is rather trivial. The case $h = 1$  cannot be reduced to the case $h = 0$, but it is straightforward to check since $\stri_{[T \| S]}$ is of the simple form
\begin{equation}\label{height:1}
\begin{tikzpicture}[xscale = .55, yscale = .55, baseline={(current bounding box.center)}]
\begin{pgfonlayer}{nodelayer}
		\node [style=none] (114) at (4.15, -1.725) {};
		\node [style=none] (119) at (7.75, -1.15) {};
		\node [style=none] (122) at (7.75, -0.925) {};
		\node [style=none] (124) at (5.475, -1.725) {};
		\node [style=none] (150) at (5.475, 0.025) {$\bullet$};
		\node [style=none] (151) at (4.15, 0.125) {};
		\node [style=small box] (152) at (8.3, -0.575) {$\kappa_{T^*}$};
		\node [style=none] (155) at (8.3, -0.35) {};
		\node [style=none] (156) at (8.3, 0.5) {};
		\node [style=none] (157) at (6.625, -1.15) {};
		\node [style=none] (158) at (6.625, 0.175) {};
		\node [style=none] (160) at (7.2, -1.925) {};
		\node [style=none] (161) at (7.2, -1.475) {};
		\node [style=none] (162) at (8.3, -1.925) {};
		\node [style=none] (163) at (8.3, -0.925) {};
		\node [style=none] (164) at (8.85, -1.15) {};
		\node [style=none] (165) at (8.85, -0.925) {};
		\node [style=none] (166) at (9.975, -1.15) {};
		\node [style=none] (167) at (9.975, 0.175) {};
		\node [style=none] (168) at (9.4, -1.875) {$\bullet$};
		\node [style=none] (169) at (9.4, -1.475) {};
		\node [style=none] (170) at (11.175, 0.025) {};
		\node [style=none] (171) at (11.175, -1.875) {$\bullet$};
	\end{pgfonlayer}
	\begin{pgfonlayer}{edgelayer}
		\draw [style=wire] (119.center) to (122.center);
		\draw [style=wire] (124.center) to (150.center);
		\draw [style=wire] (114.center) to (151.center);
		\draw [style=wire] (155.center) to (156.center);
		\draw [style=wire] (157.center) to (158.center);
		\draw [style=wire, in=-90, out=-90] (157.center) to (119.center);
		\draw [style=wire] (160.center) to (161.center);
		\draw [style=wire] (162.center) to (163.center);
		\draw [style=wire] (164.center) to (165.center);
		\draw [style=wire] (166.center) to (167.center);
		\draw [style=wire, in=-90, out=-90] (166.center) to (164.center);
		\draw [style=wire] (168.center) to (169.center);
		\draw [style=wire] (170.center) to (171.center);
	\end{pgfonlayer}
\end{tikzpicture}
\end{equation}
where $T^*$ is the set of those vertices $i \in T \mi S$ with $\dom(\kappa_i) \neq \0$ and $\kappa_{T^*} = \bigotimes_{i \in T^*} \kappa_i$; note that not all the multipliers are displayed herein.

For the inductive step $h > 1$, let $i$, $j$, and $k$ be three distinct vertices in $G_{S \rightarrow T}$ such that each is a child of the next. Let $T_j$ be as given by Lemma~\ref{graph:split}. Although $i$ does not belong to $G_{S \rightarrow T_j}$, the height of $G_{S \rightarrow T_j}$ may still be $h$. Anyway, if the heights of $G_{S \rightarrow T_j}$, $G_{T_j \rightarrow T}$ are indeed both less than $h$ then, by the inductive hypothesis,
\[
F([{T_j} \| S]_{(W, K)}) = [T_j \| S]_{(F(W), F(K))}, \quad F([{T} \| T_j]_{(W, K)}) = [T \| T_j]_{(F(W), F(K))}
\]
and hence, by Lemma~\ref{graph:split},
\[
\begin{split}
F([T \| S]_{(W, K)}) & = F([{T} \| T_j]_{(W, K)} \circ [T_j \| S]_{(W, K)})\\
                                    & = [T \| T_j]_{(F(W), F(K))} \circ [T_j \| S]_{(F(W), F(K))} \\
                                    & = [T \|  S]_{(F(W), F(K))},
\end{split}
\]
where the last equality holds because the construction of $T_j$ only depends on $G_{S \rightarrow T}$, not the diagram. If either of the heights is equal to $h$ then we can break  up the graph again in the same way, and eventually the height will have to drop.
\end{proof}

More generally, we shall consider  \memph{twists} of $[T \| S]$, that is, a morphism of the form
\[
[\tau(T) \| \sigma(S)] = \iota_{w_T \rightarrow w_{\tau(T)}} \circ [T \| S] \circ \iota_{w_{\sigma(S)} \rightarrow w_S},
\]
where $\sigma$, $\tau$ are permutations of $S$, $T$ and $\iota_{w_T \rightarrow w_{\tau(T)}}$, $\iota_{w_{\sigma(S)} \rightarrow w_S}$ are the unique multipliers generated from symmetries. Such a twist is  determined by the pair of permutations.

\begin{rem}
The  construction in Definition~\ref{gen:condi} can still be carried out without assuming straightness, albeit not canonically, because we may have to choose monoidal products if $\cat M$ is not strict and multipliers if $w^n = w^m$ could hold for $w \neq \moi$ and $n \neq m > 0$.
\end{rem}

All is not lost for an arbitrary Markov category $\cat M$. Following the discussion in \cite[\S~1.1]{joyal1991geometry}, we may first strictify  $\cat M$ into $\str(\cat M)$ as follows (this is a special case of a construction in \cite[\S~2]{maclane1985coherence}).  The objects of $\str(\cat M)$ are the words of the free monoid $W(\cat M_0)$. The \memph{evaluation} map $\eval : W(\cat M_0) \fun \cat M_0$ is given by induction on the length of the word:
\[
\eval(\0) = \moi, \quad \eval(a) = a \text{ for } a \in \cat M_0, \quad \eval(a_1 \ldots a_{n+1}) = \eval(a_1 \ldots a_{n}) \otimes a_{n+1}.
\]
The arrows $w \fun w'$ in $\str(\cat M)$ are the arrows $\eval(w) \fun \eval(w')$ in $\cat M$. The monoidal product $\bar \otimes$ in $\str(\cat M)$ is given by concatenation ${v}  \botimes {w} = vw$ and
\begin{equation*}
  \bfig
      \square(0,0)/->`->`<-`->/<1000, 400>[\eval(vw)`\eval(v'w')`\eval(v) \otimes \eval(w)`\eval(v') \otimes \eval(w'); f \botimes g`\cong`\cong`f \otimes g]
  \efig
\end{equation*}
where the vertical arrows are the unique coherence isomorphisms. There are obvious candidates for symmetries,  duplicates, and discards in $\str(\cat M)$ and  it follows from the MacLane coherence theorem that $\eval : \str(\cat M) \fun \cat M$ is indeed a Markov functor. Moreover,  $\eval$ is full and faithful and surjective on isomorphism classes, and hence is a comonoid equivalence (this is rather clear if we apply \cite[Proposition~10.16]{fritz2020synthetic}); its quasi-inverse is the obvious full embedding $\alp : \cat M \fun \str(\cat M)$, indeed $\alp$  is the right inverse of $\eval$ since $\eval \circ \alp$ is the identity.

Now we may simply define an \memph{effect} in $\cat M$ over a sequence of objects $(a_1,
\ldots, a_n)$ to be  a morphism of the form $\eval([T \| S]_{(W, K)}) : \eval(w_S) \fun \eval(w_T)$, where $\eval(w_i) = a_i$ for all $i$.

\section{Functoriality over ordered directed acyclic graphs}

In any directed acyclic graph $G$, the directed paths induce a partial ordering on $V(G)$. We say that $G$ is \memph{ordered} if $V(G)$ comes equipped with a total ordering that extends this partial ordering. The category \cat{FinODAG} has ordered directed acyclic graphs with identity loops as objects and order-preserving graph homomorphisms between them as morphisms. Let $G = (V, A, s, t, \sigma)$ be an object in \cat{FinODAG}, where the total ordering is expressed by the bijection $\sigma: V \fun \{1, \ldots, n\}$; we shall also write $V = (v_1, \ldots, v_n)$.

%A \memph{clique} in \cat{FinODAG} is a nonempty finite set $\set{G_i  \given i \in I}$ of objects together with a finite set $\set{u_{ij} : G_i \fun G_j \given (i, j) \in I \times I}$ of morphisms such that $u_{ii} = 1_{G_i}$ and $u_{ik} = u_{jk} \circ u_{ij}$, in particular, every $u_{ij}$ is an isomorphism, since $1_{G_i} = u_{ji} \circ u_{ij}$.

\begin{ter}\label{graph:bas}
Elements of the free monoid $W(V)$ are also referred to as \memph{variables}  and those of length $1$, that is, the vertices in $V$,  \memph{atomic variables}. If no atomic variable occurs more than once in a variable $v$ then $v$ is \memph{singular}; in particular, $\0$ is singular.

%A singular variable is \memph{maximal} if  each atomic variable occurs exactly once in it.
%An atomic variable $v$ is a \memph{path ancestor} of another atomic variable $w$ if there is a directed path in $G$ from $v$ to $w$, and is an \memph{ancestor} of $w$ if it is a path ancestor of $w$ or is equal to $w$. Note that $v$ is not a path ancestor of itself since a directed path is not allowed to contain identity loops. Let $v = \bigotimes_{1 \leq i \leq n} v_i$, where each $v_i$ is atomic. We say that $v_S$, $v_{S'}$ are \memph{disjoint} if no atomic variable occurs in both of them, that is,   Note that being disjoint is not the same as $w \cap w' = \0$,  unless $v$ is singular.If $v$, $w$ are singular variables then  $v \cup w$ is abbreviated as $vw$ or $wv$, which denotes the unique sub-variable of $\dot V$ that contains exactly the atomic variables in $v$, $w$.

Concatenation of two variables $v$, $w$ is also written as $v \otimes w$. As before, for each  $S \sub \{1, \ldots, n\}$ let $v_S = \bigotimes_{i \in S} v_i$, where $v_\0 = \0$. Write $v_{S'} \sub v_S$, or $v_{S'} \in v_S$ if  $S'$ is a singleton, and say that $v_{S'}$ is a \memph{sub-variable} of $v_S$ if  $S' \sub S$.  Write $v_{S} \cap v_{S'} = v_{S \cap S'}$, $v_S \mi v_{S'} = v_{S \mi S'}$, and so on.
\end{ter}

%More generally, given a permutation $\lambda$ of $\{1, \ldots, n\}$, these notational schemes may be applied to the sub-variables of $\dot V_{\lambda} = \bigotimes_{1 \leq i \leq n} v_{\lambda(i)}$.

\begin{defn}\label{cau:gen}
The free $\DD$-graph monoid \memph{$W(G)$} has $W(V)$ as its set of vertices and, for each atomic variable $v \in V$, the following  arrows added:
\[
 \0 \toleft^{\epsilon_v} v \to^{\delta_v} vv, \quad \pa(v) \to^{\kappa_v} v,
\]
where $\pa(v)$ is the singular variable  that contains exactly the parents of $v$, and is more accurately denoted by $\pa_G(v)$ if necessary. We refer to $\kappa_v$ as a \memph{causal mechanism}. If $\pa(v) = \0$ then this adds an arrow $\0 \fun v$, which is called an \memph{exogenous} causal mechanism.
\end{defn}

Denote by $\cau(G)$ the free  Markov category over $W(G)$. For convenience, set $\kappa_\0 = 1_\0$ in $\cau(G)$.

We single out a special case of Definition~\ref{gen:condi}, which is the main object of our study below.

%\begin{nota}
%First we choose a total ordering on $V$ and denote the corresponding maximal singular variable by $\dot V$. All singular variables we shall speak of are sub-variables of $\dot V$.
%
%This maneuver would be redundant if $\cau(G)$ were commutative, but that is impossible for a free Markov category. The results below depend on the chosen ordering only because taking monoidal products of atomic variables does.
%\end{nota}

\begin{defn}\label{cau:condi}
Let $v$, $w$ be  singular variables in $\cau(G)$  and $K$ the set of causal mechanisms. The $(V, K)$-effect of $v$ on $w$ in $\cau(G)$ is referred to as the \memph{causal effect} of $v$ on $w$ and is denoted simply by $[w \| v]$,  or  by $[w]$ when $v = \0$, which is  called the \memph{prior} on $w$.
\end{defn}

This notion was first introduced in \cite[\S~4]{Fong:thesis}, where it is called causal conditional.  We call it causal effect because it is closely related to the eponymous notion in the literature on causality.

Let $\phi : H \fun G$ be a morphism in $\cat{FinODAG}$. We construct a graph monoid homomorphism
\[
f_\phi : W(G) \fun \cau(H)
\]
as follows. For any variable $v \in W(G)$, denote by $\phi^{-1}(v)$  the variable $\bigotimes_{w \in v} \bigotimes_{u \in \phi^{-1}(w)} u$ in $\cau(H)$, where ``$w \in v$'' enumerates the occurrences of atomic variables in $v$, following their order in $v$,  and  ``$u \in \phi^{-1}(w)$'' enumerates the  elements of the set $\phi^{-1}(w)$, following their order in $V_H$. Clearly $\phi^{-1}(v)$ is singular if and only if $v$ is; in that case, if $v = v_S$ for some $S$ then the order in which the monoidal product $\phi^{-1}(v)$ is taken is the one induced by $V_H$, because morphisms in $\cat{FinODAG}$ preserve order. Set $f_\phi(v) = \phi^{-1}(v)$. If $v$ is atomic  then set
\begin{equation}\label{sub:fun:gen}
\begin{gathered}
f_\phi(1_v) = 1_{\phi^{-1}(v)}, \quad f_\phi(\epsilon_v ) = \epsilon_{\phi^{-1}(v)}, \quad f_\phi(\delta_v ) = \delta_{\phi^{-1}(v)}, \\
 f_\phi(\kappa_v) = [\phi^{-1}(v) \| \phi^{-1}(\pa(v))]_{(V_H, K_H)},
\end{gathered}
\end{equation}
where $1_v$ stands for the identity loop on $v$ and $K_H = \set{\kappa_{x} : x \in V_H}$.

Note that $H_{\phi^{-1}(\pa(v)) \rightarrow \phi^{-1}(v)}$ does not contain any vertex that is not already in $\phi^{-1}(\pa(v))$ or $\phi^{-1}(v)$, but this does not mean that it is of height at most $1$.

\begin{defn}[Refinement]\label{phi:inter}
Since $\cau(G)$ is free, the graph monoid homomorphism $f_\phi$ induces a  Markov functor,  called the \memph{$\phi$-refinement}:
\[
\phi^* : \cau(G) \fun \cau(H).
\]

\end{defn}

In principle, $\phi^*$ is unique up to a unique  monoidal natural isomorphism. Here it is simply unique because it must be strict (see the remark after Theorem~\ref{free:markov:quot}).

Heuristically, we think of $\phi$ as a substitution operation because the vertices of the graphs are thought of as the atomic variables in the corresponding ``syntactic categories'' for causal inference.

\begin{exam}\label{trea:two}
Let $v \in \cau(G)$ be a singular variable and $G_{\bar v}$ be the subgraph of $G$ with all the incoming arrows at $v$ removed. Let $i_v : G_{\bar v} \fun G$ be the obvious graph embedding. We refer to the $i_v$-refinement
\[
i_{v}^* :  \cau(G) \fun \cau(G_{\bar v})
\]
as the \memph{$i_v$-intervention}. So, the only nontrivial assignment in (\ref{sub:fun:gen}) for the construction of $i_{v}^*$ is, for each $u \in v$,
\begin{equation}\label{inter:exam}
\begin{tikzpicture}[xscale = .5, yscale = .55, baseline=(current  bounding  box.center)]
\begin{pgfonlayer}{nodelayer}
		\node [style=none] (210) at (1.9, 4.3) {};
		\node [style=none] (211) at (1.9, 3.55) {};
		\node [style=wide small box] (225) at (1.9, 3.025) {$\kappa_u$};
		\node [style=none] (226) at (1.175, 2.575) {};
		\node [style=none] (227) at (1.175, 1.775) {};
		\node [style=none] (228) at (2.625, 2.575) {};
		\node [style=none] (229) at (2.625, 1.775) {};
		\node [style=none] (286) at (4.625, 2.925) {$\efun$};
		\node [style=none] (295) at (1.97, 2) {$\cdots$};
		\node [style=none] (296) at (7.2, 4.3) {};
		\node [style=none] (297) at (7.2, 3.55) {$\bullet$};
		\node [style=none] (299) at (6.475, 2.575) {$\bullet$};
		\node [style=none] (300) at (6.475, 1.775) {};
		\node [style=none] (301) at (7.925, 2.575) {$\bullet$};
		\node [style=none] (302) at (7.925, 1.775) {};
		\node [style=none] (304) at (7.27, 2) {$\cdots$};
		\node [style=none] (305) at (4.625, 0.425) {$i_{v}^*(\kappa_u) = \bar \kappa_u \otimes \epsilon_{\pa(u)}$};
	\end{pgfonlayer}
	\begin{pgfonlayer}{edgelayer}
		\draw [style=wire] (210.center) to (211.center);
		\draw [style=wire] (226.center) to (227.center);
		\draw [style=wire] (228.center) to (229.center);
		\draw [style=wire] (210.center) to (211.center);
		\draw [style=wire] (226.center) to (227.center);
		\draw [style=wire] (228.center) to (229.center);
		\draw [style=wire] (296.center) to (297.center);
		\draw [style=wire] (299.center) to (300.center);
		\draw [style=wire] (301.center) to (302.center);
		\draw [style=wire] (296.center) to (297.center);
		\draw [style=wire] (299.center) to (300.center);
		\draw [style=wire] (301.center) to (302.center);
	\end{pgfonlayer}
\end{tikzpicture}
\end{equation}
where $\kappa_u$, $\bar \kappa_u$ are the causal mechanisms on $u$ in $\cau(G)$, $\cau(G_{\bar v})$, respectively.

In the definition of $\cau(G)$, we do not have an exogenous causal mechanism $\0 \fun w$ for every $w \in V_G$ unless $w$ has no parents in $G$. Of course, such generators, indeed, arbitrary generators, can be introduced to suit the task at hand. In that case,  there is no need to go through this subgraph $G_{\bar v}$ anymore, since (\ref{inter:exam}) can already be defined within $\cau(G)$ itself so that $i_{v}^*$ becomes an endofunctor $\cau(G) \fun \cau(G)$. This is the approach adopted  in \cite{jacobs2019causal}, where a specific interpretation is intended for these exogenous causal mechanisms, namely uniform probability distribution. Casting $i_{v}^*$ as an endofunctor is merely for convenience and elegance, though, not a technical necessity.
\end{exam}

\begin{prop}\label{cond:preserve}
Let $v$, $w$ be singular variables in $\cau(G)$. Then $\phi^*([w \| v]) = [\phi^*(w) \| \phi^*(v)]$.
\end{prop}

This assertion is a bit terse  and may seem to be just a consequence of Proposition~\ref{mar:fun:eff}. But the situation   becomes clearer if we do not abbreviate the notation:
\begin{equation}\label{comp:two}
\bfig \morphism(0,0)/=/<1450,0>[[\phi^*(w) \| \phi^*(v)]_{( \phi^*(V_G), \phi^*(K_G))}`\phi^*([w \| v]_{(V_G, K_G)}); \textup{Proposition~\ref{mar:fun:eff}}]
\morphism(1450,0)/=/<1330,0>[\phi^*([w \| v]_{(V_G, K_G)})`[\phi^*(w) \| \phi^*(v)]_{(V_H, K_H)}; \textup{Proposition~\ref{cond:preserve}}]
\efig
\end{equation}
where $K_G = \set{\kappa_{u} : u \in V_G}$.

\begin{proof}
We shall indeed show that the first and last effects in (\ref{comp:two}) are equal. To that end, let $\Gamma_{\phi^*G}$ be the diagram constructed from the diagram $\Gamma_{[\phi^*(w) \| \phi^*(v)]_{( \phi^*(V_G), \phi^*(K_G))}}$ by depicting each relevant $\phi_*(\kappa_y)$ as $\Gamma_{[\phi^{-1}(y) \| \phi^{-1}(\pa(y))]_{(V_H, K_H)}}$ and then rewriting the multipliers accordingly. After surgeries, we assume that $\Gamma_{\phi^*G}$ is Markov minimal. Abbreviate $\Gamma_{[\phi^*(w) \| \phi^*(v)]_{(V_H, K_H)}}$ as $\Gamma_{H}$; note that $\Gamma_{H}$ is already Markov minimal. Therefore, it is enough to show that $\Gamma_{\phi^*G}$ is Markov congruent to $\Gamma_{H}$.

If $p$ is a directed path in $H_{\phi^*(v) \rightarrow \phi^*(w)}$ that ends at some vertex $x \in V_H$ then  $\phi(p)$ is a directed path in $G_{v \rightarrow w}$ that ends at  $\phi(x) \in V_G$ and does not travel toward $v$, unless $p$ is completely collapsed by $\phi$, in which case $\phi(p)$ is the identity loop on $\phi(x)$. So $\phi(H_{\phi^*(v) \rightarrow \phi^*(w)})$ is a subgraph of $G_{v \rightarrow w}$. This implies that $\Gamma_{\phi^*G}$ has a unique node with the value $\kappa_x$ for every vertex $x \in V(H_{\phi^*(v) \rightarrow \phi^*(w)}) \mi \phi^*(v)$. Since  $\Gamma_{\phi^*G}$ is Markov minimal, we see that if every multiplier involved is depicted as a diagram with at most one node (recall the last sentence of Remark~\ref{mul:pow}) then there is a bijection between the maximal directed paths  in $\Gamma_{\phi^*G}$ and those in $\Gamma_{H}$ such that each matching pair thread through edges and nodes with the same values. It follows that $\Gamma_{\phi^*G}$ is Markov congruent to $\Gamma_{H}$.
\end{proof}

%and the multiplier in the component $\Gamma_u$ of $\Gamma_{H}$ is Markov congruent to a subdiagram of the multiplier in the component $\Gamma_u$ of $\Gamma_{\phi^*G}$. In fact, if the latter multiplier had an extra node then every maximal directed path passing through it in $\Gamma_{\phi^*G}$ must end at a terminal node, for otherwise, it would end in $\phi^*(w)$ and hence must be in the subdiagram in question.
%
%
%It is not hard to see that the is  used in are  there is an embedding $\phi: \stri'_H \fun \stri_{\phi^*G}$, where $\stri'_H$ is Markov congruent to $\Gamma_{H}$. On the other hand, since we may assume that, in its construction as a diagram of an effect, each  component $\stri_{\phi^*(u)}$ is a composition of a multiplier and a diagram of the form (\ref{height:1}), but without terminal nodes, unless $u \in v$. Therefore,

\begin{thm}[Functoriality of refinement]\label{int:func}
The construction of $\phi^*$ is functorial, that is, there is a contravariant  functor
\[
\intv : \cat{FinODAG} \fun \cat{MarCat}
\]
sending each object  $G \in  \cat{FinODAG}$ to $\cau(G)$ and each morphism $\phi \in  \cat{FinODAG}$ to $\phi^*$.
\end{thm}
\begin{proof}
We need to check that, for all morphisms $G'' \to^\psi G' \to^\phi G$ in $\cat{FinODAG}$,  $\psi^* \circ \phi^* = (\phi \circ \psi)^*$ as symmetric strict monoidal functors. For objects, this follows from the fact that  morphisms in $\cat{FinODAG}$ preserve order. For morphisms, since the categories are free, it is enough to check that, for all generators $g \in \cau(G)$, $\psi^*(\phi^*(g)) = (\psi \circ \phi)^*(g)$. This is clear if $g$ is  $1_v$, $\epsilon_v$, or $\delta_v$. The case $g = \kappa_v$ follows from Proposition~\ref{cond:preserve}.
\end{proof}

%---------------------------------------------------------------------
%Included for Gather Purpose only:
%input "C:\localtexmf\bibtex\bib\mybib\MYbib.bib"
%\bibliographystyle{amsplain}
%\bibliography{C:/Users/yimuy/Dropbox/causalityandcounterfactual/cau_add_ref}

\begin{thebibliography}{1}

\bibitem{Fong:thesis}
Brendan Fong, \emph{Causal theories: {A} categorical perspective on {B}ayesian
  networks},  (2013), arXiv:1301.6201.

\bibitem{fritz2020synthetic}
Tobias Fritz, \emph{A synthetic approach to {M}arkov kernels, conditional
  independence and theorems on sufficient statistics}, Advances in Mathematics
  \textbf{370} (2020), 107239.

\bibitem{FriLiang2022}
Tobias Fritz and Wendong Liang, \emph{Free {CD} categories and free {M}arkov
  categories}, arXiv:2204.02284.

\bibitem{jacobs2019causal}
Bart Jacobs, Aleks Kissinger, and Fabio Zanasi, \emph{Causal inference by
  string diagram surgery}, International conference on foundations of software
  science and computation structures, Springer, 2019, pp.~313--329.

\bibitem{joyal1991geometry}
Andr{\'e} Joyal and Ross Street, \emph{The geometry of tensor calculus, {I}},
  Advances in mathematics \textbf{88} (1991), no.~1, 55--112.

\bibitem{maclane1985coherence}
Saunders MacLane and Robert Par{\'e}, \emph{Coherence for bicategories and
  indexed categories}, Journal of Pure and Applied Algebra \textbf{37} (1985),
  59--80.

\bibitem{selinger2010survey}
Peter Selinger, \emph{A survey of graphical languages for monoidal categories},
  New structures for physics, Springer, 2010, pp.~289--355.

\bibitem{yinzhang2021}
Yimu Yin and Jiji Zhang, \emph{{M}arkov categories, causal theories, and the
  do-calculus}, Studies in Logic \textbf{14} (2021), 1--24.

\end{thebibliography}
%---------------------------------------------------------------------

\providecommand{\bysame}{\leavevmode\hbox to3em{\hrulefill}\thinspace}
\providecommand{\MR}{\relax\ifhmode\unskip\space\fi MR }
% \MRhref is called by the amsart/book/proc definition of \MR.
\providecommand{\MRhref}[2]{%
  \href{http://www.ams.org/mathscinet-getitem?mr=#1}{#2}
}
\providecommand{\href}[2]{#2}

\end{document}